\documentclass[12pt,a4paper]{article}
\usepackage{fullpage}
\usepackage[utf8]{inputenc}
\usepackage{amsfonts}
\usepackage{amsmath}
\usepackage{amssymb}
\usepackage{amsthm}
\usepackage{amscd}   
\usepackage{hyperref}
\usepackage[usenames]{color}
\usepackage{tikz}
\usetikzlibrary{calc,intersections,through,backgrounds,arrows}

\usepackage[refpage]{nomencl}

\newtheorem{theorem}{Theorem}[section]
\newtheorem{lemma}[theorem]{Lemma}
\newtheorem{proposition}[theorem]{Proposition}

\newtheorem{corollary}[theorem]{Corollary}

\theoremstyle{definition}
\newtheorem{definition}[theorem]{Definition}
\newtheorem{example}[theorem]{Example}
\newtheorem{remark}[theorem]{Remark}
\newtheorem{notation}[theorem]{Notation}

\renewcommand\le{\leqslant}
\renewcommand\ge{\geqslant}

\def\bbZ{\mathbb Z}

\def\bbP{\mathbb P}

\def\mcA{\mathcal A}
\def\mcB{\mathcal B}
\def\mcC{\mathcal C}
\def\mcU{\mathcal U}
\def\mcV{\mathcal V}

\def\mcZ{\mathcal Z}

\def\a{\mathfrak a}

\def\b{\mathfrak b}

\def\Id{\mathrm{Id}}

\def\Par{\mathrm{Par}\,}
\def\Zero{\mathbf{0}}
\def\One{\mathbf{1}}

\def\SET{\mathrm{SET}}
\def\SEQ{\mathrm{SEQ}}
\def\CYC{\mathrm{CYC}}

\def\mcE{\mathcal E}

\def\mcL{\mathcal L}       

\def\mcP{\mathcal P}      
\def\mcCP{\mathcal{C}}   
\def\p{\mathfrak p}       

\def\mcF{\mathcal F}      

\def\mcG{\mathcal G}      
\def\mcCG{\mathcal{CG}}   
\def\mcT{\mathcal T}      
\def\mcIT{\mathcal{IT}}   
\def\it{\mathfrak{it}}    

\def\mcD{\mathcal D}      
\def\d{\mathfrak d}       

\def\mcQSS{\mathcal{QSS}} %
\def\qss{\mathfrak {qss}} %

\def\mcS{\mathcal S}        
\def\mcQSS{\mathcal{QSS}}   
\def\qss{\mathfrak {qss}}   
\def\mcCQSS{\mathcal{CQSS}} 
\def\mcPS{\mathcal{PS}}     
\def\ps{\mathfrak {ps}}     
\def\mcCPS{\mathcal{CPS}}   

\def\mcGEM{\mathcal{GEM}}   
\def\gem{\mathfrak {gem}}   
\def\mcCGEM{\mathcal{CGEM}} 

\author{Thierry Monteil\footnote{
    LIPN, CNRS (UMR 7030), Universit\'e Paris 13, F-93430 Villetaneuse, France.\newline
    Email: \texttt{thierry.monteil@lipn.univ-paris13.fr}}
    \and
    Khaydar Nurligareev\footnote{
    LIP6, CNRS (UMR 7606), Sorbonne Universit\'e, F-75005 Paris, France.\newline
    LIB, Université de Bourgogne, F-21078 Dijon, France.\newline
    LIPN, CNRS (UMR 7030), Universit\'e Paris 13, F-93430 Villetaneuse, France.\newline
    Email: \texttt{khaydar.nurligareev@lip6.fr}}
  }
\date{}

\title{Asymptotic probability of irreducibles III: Anti-SEQ}

\begin{document}

\maketitle

\begin{abstract}
In this paper, we study the structure of the complete asymptotic expansion of the probability that a large combinatorial object is connected or consists of a~given number of connected components.
For rapidly growing 
labeled families of structures, the coefficients involved in these expansions are possibly negative integers.
Using species theory, we interpret these coefficients as the difference between the counting sequences of two derivative species of structures.
In particular, we show that this difference can be viewed as the counting sequence of the virtual species obtained with the help of an ``anti-$\SEQ$'' operator applied to the initial family of structures.
Applications include $P$-angulated discrete surfaces, quadratic square-tiled surfaces, and non-orientable graph encoded manifolds, which were not reachable with our previous methods.

Moving on to the weighted species, we establish the whole structure of the asymptotic expansion of the probability that a~graph is connected in the Erd\H{o}s-Rényi model $G(n,p)$.
Here, the asymptotic coefficients are polynomials in $\frac{p}{1-p}$ and can be described both in terms of simple graphs and irreducible  tournaments with ties.
We also provide general asymptotic results for sequence and cycle decomposition, as well as the complete asymptotic expansion of the probability that a random labeled tournament with ties is irreducible.
\end{abstract}

\section{Introduction}

This paper belongs to a series that aim to provide precise bounds on the probability that a combinatorial object is irreducible, for some notions of irreducibility.

In~\cite{MonteilNurligareevSET}, we provided a method to describe the whole structure of the asymptotic expansion of the probability that a large combinatorial object is connected.
For instance, for simple graphs, we showed that the asymptotic probability that a random labeled simple graph $g$ with $n$~vertices is connected satisfies
  \begin{equation*}
    \mathbb{P}\big(g\mbox{ is connected}\big)
     \approx 1 - 
    \sum\limits_{k\ge1}
     \it_k \cdot
     \binom{n}{k} \cdot
     \frac{2^{k(k+1)/2}}{2^{kn}},
  \end{equation*}
where $\it_k$ denotes the number of irreducible tournaments of size $k$.
In particular, the first terms of this asymptotic expansion are
  $$
    \mathbb{P}\big(g\mbox{ is connected}\big)
     = 1 - 
    \binom{n}{1}\frac{1}{2^{n-1}} - 
    2\binom{n}{3}\frac{1}{2^{3n-6}} - 
    24\binom{n}{4}\frac{1}{2^{4n-10}} + 
    O\left(\frac{n^5}{2^{5n}}\right).
  $$

In this example, the emergence of the derivative structure (irreducible tournaments) with respect to the first structure (connected graphs) can be explained as follows.
Although the labeled classes of simple graphs and tournaments are not isomorphic, they have the same counting sequence $2^{\binom{n}{2}}$.
From an enumerative perspective, this fact leads to a double decomposition: a graph can be decomposed into a~set of its connected components (connected graphs), while a tournament can be decomposed into a sequence of its strongly connected components (irreducible tournaments).
Thus, the counting sequence of irreducible tournaments acts as a derivative sequence with respect to the counting sequence of connected graphs.

More generally, the following result holds~\cite{MonteilNurligareevSET}.

\begin{theorem}
\label{theorem:SET-asymptotics}
  Let $\mcA$ be a gargantuan labeled combinatorial class that admits a double decomposition $\mcA = \SET(\mcB) = \SEQ(\mcD)$ for some labeled combinatorial classes $\mcB$ and~$\mcD$.
  Suppose that $a\in\mcA$ is a random object of size~$n$.
  In this case,
  \begin{equation*}
    \mathbb{P}(a\in\mcB)
     \approx 1 - 
    \sum\limits_{k\ge1}
     \d_k \cdot \binom{n}{k} \cdot \frac{\a_{n-k}}{\a_n},
  \end{equation*}
  where $(\a_n)$ and $(\d_n)$ are the counting sequences of the classes $\mcA$ and $\mcD$, respectively.
\end{theorem}

Theorem~\ref{theorem:SET-asymptotics} is seen to apply in the setting where labeled combinatorial classes admit a double $\SET/\SEQ$ decomposition.
There are a number of combinatorial classes that admit such a double $\SET/\SEQ$ decomposition in a natural way (examples include graphs, abelian square-tiled surfaces, combinatorial maps, and constellations; see~\cite{MonteilNurligareevSET}), but some classes seem not to have one.
For example, we can prove that the asymptotic expansion of the probability $p_n$ that a labeled triangulated surface \cite{BrooksMakover2004, PippengerSchleich2005} is connected satisfies
  \begin{align*}
    p_n = 1 & -
    \binom{2n}{2}\dfrac{15}{(6n-1)(6n-3)(6n-5)} -
    \binom{2n}{4}\dfrac{9\,045}{(6n-1)\ldots(6n-11)} \\ & -
    \binom{2n}{6}\dfrac{30\,085\,425}{(6n-1)\ldots(6n-17)} - \ldots -
    \d_{2k}\binom{2n}{2k}\dfrac{\big(6(n-k)-1\big)!!}{(6n-1)!!} - \ldots
  \end{align*}
However, we are not able to interpret the integer sequence
  \[
    (\d_n) = 15,\, 9\,045,\, 30\,085\,425,\, 282\,543\,711\,975, \ldots
  \]
as a counting sequence of any labeled combinatorial class.
Moreover, this sequence does not appear in the OEIS (\cite{OEIS2026}).

The goal of this paper is to provide such a combinatorial interpretation, with the additional cost that the sequence $(\d_n)$ is not a counting sequence of a combinatorial class, but the difference of two such counting sequences.
To this end, given the decomposition of a combinatorial class $\mcA$ as the $\SET$ of its connected components, we construct an ``anti-$\SEQ$'' operator that produces a virtual class $\mcD$ satisfying $\mcA=\SEQ(\mcD)$ and thus provides the desired combinatorial interpretation.
We call the class $\mcD$ virtual here because, in certain cases, some of the coefficients of its ``counting sequence'' are negative, and so it cannot be interpreted as a counting sequence of any true combinatorial class. 

At this point, we switch from the symbolic method to species theory, whose language is well-adapted for such interpretations.
Within this framework, one replaces combinatorial constructions with substitutions.
In particular, the constructions $\SET$ and $\SEQ$ are to be replaced with compositions with species $\mcE$ of sets and $\mcL$ of linear orders, respectively.
What allows us to reach our goals is the concept of virtual species.
In a few words, virtual species is a way to establish a combinatorial form of subtraction, which is similar to extending natural numbers to integers:
combinatorial objects are allowed to be taken both positive and negative.
This concept is also helpful for building inverses for multiplications and substitutions. 
In particular, the relation $\mcA=\mcL\circ\mcD$ can be inverted as $\mcD = \mcL_+^{(-1)}\circ\mcA_+$, so that we obtain an ``anti-$\SEQ$'' operator as the inverse $\mcL_+^{(-1)}$.
To jump ahead a bit, the desired result can then be stated in the following form.

{
\renewcommand{\thetheorem}{\ref{prop: d_(k,1) interpretation}}

\begin{proposition}
 Let $\mcA$ be a gargantuan species of structures that can be represented as $\mcA=\mcE\circ\mcB$ for some species of structures~$\mcB$.
 Assume that $n$ is a positive integer and $s\in\mcA$ is a random $\mcA$-structure on $[n]$.
 In this case,
 \begin{equation*}
  \bbP(s\mbox{ is connected})
   \approx 
  1 - \sum\limits_{k\ge1}\d_k\cdot\binom{n}{k}\cdot\dfrac{\a_{n-k}}{\a_n},
 \end{equation*}
 where $\a_n$ is the counting sequence of $\mcA$,
 \begin{equation*}
  \d_{k}
   = 
  \#\{s'\in\mcA_k \mid \pi_0(s') \mbox{ is odd}\}
   -
  \#\{s'\in\mcA_k \mid \pi_0(s') \mbox{ is even}\}
\end{equation*}
 and $\pi_0(s')$ is the number of connected components of the structure $s'$.
\end{proposition}
\addtocounter{theorem}{-1}
}

As a particular case, for the probability discussed above that a labeled triangulated surface is connected, we obtain the asymptotic expansion of the form
  \begin{equation*}
    \mathbb{P}\big(s\mbox{ is connected}\big)
     \approx 1 - 
    \sum\limits_{k\ge1}
     \d_{k} \cdot
     \binom{2n}{2k} \cdot
     \dfrac{\big(6(n-k)-1\big)!!}{(6n-1)!!}
  \end{equation*}
with
\[
 \d_{k}
  = 
 \#\{s'\in\mcPS_{2k}(3) \mid \pi_0(s') \mbox{ is even}\}
  -
 \#\{s'\in\mcPS_{2k}(3) \mid \pi_0(s') \mbox{ is odd}\},
\]
where $\mcPS_{2k}(3)$ is the labeled class of triangulated surfaces glued with $2k$ triangles and $\pi_0(s')$ is the number of connected components of the surface $s'$ (see Proposition~\ref{prop: ps(p) asymptotics}).

\

In fact, the approach we are developing can be applied to a wider range of cases where the structure enumeration is performed according to specific parameters.
For some parameters, such as the number of connected components, it suffices to pass to the restriction to subspecies: while species $\mcB$ can be interpreted as connected subspecies of species $\mcA=\mcE\circ\mcB$, species $\mcB^{\{m\}}=\mcE_m\circ\mcB$ represent structures consisting of exactly $m$ connected components.
Other parameters can be encoded with certain weights, which leads to the concept of weighted species: now each structure is counted not once, but based on its weight relative to the parameter in question.
All the above said is represented by the following general result.

{
\renewcommand{\thetheorem}{\ref{theorem: species E asymptotics}}
\begin{theorem}
 Let $\mcA$ be a gargantuan (weighted) species of structures that can be represented as $\mcA=\mcE\circ\mcB$ for some (weighted) species of structures~$\mcB$.
 Assume that $m,n\in\bbZ_{\geqslant0}$, and $s\in\mcA$ is a random $\mcA$-structure on $[n]$.
 In this case,
 \[
  \bbP(s\in\mcB^{\{m\}})
   \approx 
  \sum\limits_{k\ge0}\d_{k,m}\cdot\binom{n}{k}\cdot\dfrac{\a_{n-k}}{\a_n},
 \]
 where $\a_n$ and $\d_{n,m}$ are the total weights on the set $[n]$ of the species of structures $\mcA$ and $\mcD(m) = \mcB^{\{m-1\}}(\mcE^{-1}\circ\mcB)$, respectively.
\end{theorem}
\addtocounter{theorem}{-1}
}

Due to Lemma~\ref{lemma: magical species identity}, we can claim that the ``anti-$\SEQ$'' operator satisfies
 \[
  \mcL_+^{(-1)} \equiv \One - \mcE^{-1}\circ\mcE_+^{(-1)}.
 \]
Therefore, in the case where $m=1$, Theorem~\ref{theorem: species E asymptotics} indeed provides an asymptotic generalization of Theorem~\ref{theorem:SET-asymptotics} in terms of the ``anti-$\SEQ$'' operator.

\

One of the new applications that cannot be treated using previously known methods is the Erd\H{o}s--Rényi model $G(n,p)$ of graphs~\cite{ErdosRenyi1959,Gilbert1959}.
According to this model, a graph with $n$ vertices is constructed randomly: each possible edge is drawn independently with probability $p\in(0,1)$.
Gilbert~\cite{Gilbert1959} showed that the asymptotic probability to obtain a~connected graph within $G(n,p)$ satisfies
 \[
  \bbP(G\mbox{ is connected}) =
  1 - nq^{n-1} +
   O\big(n^2q^{3n/2}\big),
 \]
where $q=1-p$.
Introducing the weight of a graph $g$ as $w(g)=\rho^{\ell}$ where $\ell$ is the number of edges in $g$ and $\rho=p/q$,
with the help of Theorem~\ref{theorem: species E asymptotics} we establish the following more general result.

{
\renewcommand{\thetheorem}{\ref{theorem: Erdos-Renyi asymptotics}}
\begin{theorem}
  Let $m$ be a fixed positive integer, and $p\in(0,1)$.
 The asymptotic probability that a random graph~$g$ in the Erd\H{o}s--Rényi model~$G(n,p)$ has $m$ connected components satisfies 
 \[
  \bbP(g\mbox{ has }m\mbox{ connected components})
   \approx
  \sum\limits_{k\ge0} P_{k,m}(\rho) \cdot \binom{n}{k} \cdot \dfrac{q^{nk}}{q^{k(k+1)/2}},
 \]
 where
 \[
  P_{k,m}(\rho) = \sum\limits_{g'\in\mcG_k}(-1)^{\pi_0(g')-(m-1)}\binom{\pi_0(g')}{m-1}w(g')
 \]
 and $\pi_0(g')$ is the number of connected components of the graph $g'$.
\end{theorem}
\addtocounter{theorem}{-1}
}

According to our previous work~\cite{MonteilNurligareev2021,MonteilNurligareevSET}, in the particular case $m=1$ and $p=q=1/2$ the coefficients $-P_{k,1}(1)$ can be interpreted as the counting sequence of irreducible tournaments.
We show that this interpretation can be extended to the general case $p\in(0,1)$, so that polynomials $-P_{k,1}(\rho)$ count irreducible tournaments with ties (Theorem~\ref{theorem: Erdos-Renyi asymptotics in terms of tournaments}).
Thus, each polynomial receives two different interpretations.
One of them is in terms of weights of graphs, for example,
\begin{align*}
 -P_{3,1}(\rho)
  &=
 \rho^3 + 3\rho^2 - 3\rho + 1 \\
  &=
 w \negthickspace \left(\begin{tikzpicture}
  \coordinate (a1) at (0,10pt);
  \coordinate (a2) at ([rotate = 120] a1);
  \coordinate (a3) at ([rotate = 240] a1);
  \draw (a1) -- (a2) -- (a3) -- (a1);
  \foreach \p in {a1,a2,a3} \filldraw [black] (\p) circle (1.5pt);
 \end{tikzpicture}\right)
  +
 3 \thinspace w \negthickspace \left(\begin{tikzpicture}
  \coordinate (a1) at (0,10pt);
  \coordinate (a2) at ([rotate = 120] a1);
  \coordinate (a3) at ([rotate = 240] a1);
  \draw (a2) -- (a1) -- (a3);
  \foreach \p in {a1,a2,a3} \filldraw [black] (\p) circle (1.5pt);
 \end{tikzpicture}\right)
  -
 3 \thinspace w \negthickspace \left(\begin{tikzpicture}
  \coordinate (a1) at (0,10pt);
  \coordinate (a2) at ([rotate = 120] a1);
  \coordinate (a3) at ([rotate = 240] a1);
  \draw (a2) -- (a3);
  \foreach \p in {a1,a2,a3} \filldraw [black] (\p) circle (1.5pt);
 \end{tikzpicture}\right)
  +
 w \negthickspace \left(\begin{tikzpicture}
  \coordinate (a1) at (0,10pt);
  \coordinate (a2) at ([rotate = 120] a1);
  \coordinate (a3) at ([rotate = 240] a1);
  \foreach \p in {a1,a2,a3} \filldraw [black] (\p) circle (1.5pt);
 \end{tikzpicture}\right).
\end{align*}
 Another one is in terms of weights of irreducible tournaments with ties, such as
\begin{align*}
 -P_{3,1}(\rho)
  &=
 (\rho-1)^3 + 6(\rho-1)^2 + 6(\rho-1) + 2 \\
  &=
 w \negthickspace \left(\begin{tikzpicture}[>=latex]
  \coordinate (a1) at (0,10pt);
  \coordinate (a2) at ([rotate = 120] a1);
  \coordinate (a3) at ([rotate = 240] a1);
  \draw[<->] (a1) -- (a3);
  \draw[<->] (a2) -- (a1);
  \draw[<->] (a2) -- (a3);
  \foreach \p in {a1,a2,a3} \filldraw [black] (\p) circle (1.5pt);
 \end{tikzpicture}\right)
  +
 6 \thinspace w \negthickspace \left(\begin{tikzpicture}[>=latex]
  \coordinate (a1) at (0,10pt);
  \coordinate (a2) at ([rotate = 120] a1);
  \coordinate (a3) at ([rotate = 240] a1);
  \draw[<->] (a1) -- (a3);
  \draw[<->] (a2) -- (a1);
  \draw[->] (a2) -- (a3);
  \foreach \p in {a1,a2,a3} \filldraw [black] (\p) circle (1.5pt);
 \end{tikzpicture}\right)
  +
 6 \thinspace w \negthickspace \left(\begin{tikzpicture}[>=latex]
  \coordinate (a1) at (0,10pt);
  \coordinate (a2) at ([rotate = 120] a1);
  \coordinate (a3) at ([rotate = 240] a1);
  \draw[->] (a1) -- (a3);
  \draw[->] (a2) -- (a1);
  \draw[<->] (a2) -- (a3);
  \foreach \p in {a1,a2,a3} \filldraw [black] (\p) circle (1.5pt);
 \end{tikzpicture}\right)
  +
 2 \thinspace w \negthickspace \left(\begin{tikzpicture}[>=latex]
  \coordinate (a1) at (0,10pt);
  \coordinate (a2) at ([rotate = 120] a1);
  \coordinate (a3) at ([rotate = 240] a1);
  \draw[->] (a1) -- (a3);
  \draw[->] (a2) -- (a1);
  \draw[<-] (a2) -- (a3);
  \foreach \p in {a1,a2,a3} \filldraw [black] (\p) circle (1.5pt);
 \end{tikzpicture}\right)
\end{align*}
(here the weight of a tournament $t$ is $w(t)=(\rho-1)^{\ell}$, where $\ell$ is the number of ties in $t$).
From an algebraic point of view, these two interpretations correspond to the decomposition of polynomials $P_{k,1}(\rho)$ into two different bases.

Note that if $p<1/2$, then certain tournaments with ties are purely virtual, that is, their weights are negative.
Moreover, the total weight of irreducible tournaments on certain sets is also negative in this case.
For example, for $p=1/4$, the asymptotic probability that a random graph $g$ is connected satisfies
 \begin{multline*}
  \bbP(g\mbox{ is connected})
   = \\
  1
   -
  \binom{n}{1}\left(\dfrac{3}{4}\right)^{n-1}
   +
  \dfrac{2}{3}\binom{n}{2}\left(\dfrac{3}{4}\right)^{2n-3}
   -
  \dfrac{10}{27}\binom{n}{3}\left(\dfrac{3}{4}\right)^{3n-6}
   +
  \dfrac{8}{729}\binom{n}{4}\left(\dfrac{3}{4}\right)^{4n-10}
   +
  \ldots
 \end{multline*}
This fact certifies that our result cannot be obtained within the symbolic method and that the use of species theory is crucial.

\

Finally, Theorem~\ref{theorem: species E asymptotics} can be adapted for $\SEQ$ and $\CYC$ decompositions that, in the language of the species, are translated as compositions with the species $\mcL$ of linear orders and $\mcCP$ of cycles, respectively (Theorem~\ref{theorem: species L asymptotics}).
Thus, we obtain the asymptotic probability that a random tournament with ties is irreducible or consists of a fixed number of irreducible parts (Theorem~\ref{theorem: tournaments with ties asymptotics}).

\

The paper is organized as follows.
In Section~\ref{sec: tools}, we introduce useful notations and recall the necessary concepts from species theory, as well as some properties of gargantuan sequences and Bender's theorem.
In Section~\ref{sec: main result}, we establish our main result, Theorem~\ref{theorem: species E asymptotics} and discuss its combinatorial interpretation.
We also provide adaptations for the asymptotics related to species $\mcL$ and $\mcCP$ (Theorem~\ref{theorem: species L asymptotics}), as well as adaptations for the case where the counting sequences are $p$-periodic (Proposition~\ref{prop: p-periodic species asymptotics}).
The next part is devoted to applications.
More precisely, in Section~\ref{subsec: Erdos-Renyi} we explore the asymptotic probability of connected graphs and irreducible tournaments with ties within the Erd\H{o}s--Rényi model,
while Section~\ref{subsec: surface models} is devoted to the asymptotic probability of quadratic square-tiled surfaces (Proposition~\ref{prop: qss asymptotics}), $P$-angulated surfaces (Proposition~\ref{prop: ps(p) asymptotics}), and graph encoded manifolds (Proposition~\ref{prop: gem asymptotics}).
Finally, we conclude the paper by Section~\ref{sec: conclusion} with a number of open problems.

\section{Tools}
\label{sec: tools}
 
\subsection{Notation}
\label{subsec: notation}

\begin{notation}\label{notation: approx} 
 For a sequence $(a_n)$ and an integer $m$, we write
  \begin{equation}\label{formula: asymptotic expansion}
   a_n \approx \sum\limits_{k\ge m}f_k(n)
  \end{equation}
  if for every integer $r\ge m$,
  \[
   a_n = \sum\limits_{k=m}^{r}f_k(n) + O\big(f_{r+1}(n)\big),
  \]
  and for every integer $k\ge m$,
  \[
   f_{k+1}(n) = o\big(f_k(n)\big).
  \]
 The expression \eqref{formula: asymptotic expansion} is called an \emph{asymptotic expansion} of the sequence $(a_n)$.
\end{notation}

\subsection{Species of structures}
\label{subsec: species}

In this section, we recall the basic notions of the theory of combinatorial species introduced in 1981 by Joyal~\cite{Joyal1981}.
In order to avoid excessive abstraction, we do this in the spirit of the book written by Bergeron, Labelle, and Leroux~\cite{BergeronLabelleLeroux1998}.
Another introductory reference that the reader may find useful is~\cite[Chapter 4]{Nurligareev2022}.

\subsubsection{The concept of species of structures}

A \emph{(combinatorial) species of structures} is a rule $F$ such that
 \begin{enumerate}
  \item for each finite set $U$, the rule $F$ produces a finite set $F[U]$,
  \item for each bijection $\sigma\colon U\to V$, the rule $F$ produces a map
  \[
   F[\sigma]\colon F[U]\to F[V],
  \]
  satisfying the following properties:
  \begin{enumerate}
   \item for all bijections $\sigma\colon U\to V$ and $\tau\colon V\to W$,
    \[
     F[\tau\circ\sigma] = F[\tau]\circ F[\sigma];
    \]
   \item for the identity map $\Id_{U}\colon U\to U$,      \[
     F[\Id_U] = \Id_{F[U]}.
    \]
  \end{enumerate}
 \end{enumerate}
The elements of the set $F[U]$ are called $F$-\emph{structures on} $U$, or \emph{structures of} $F$ \emph{on} $U$,
while the map $F[\sigma]$ is called the \emph{transport function of} $F$-\emph{structures along} $\sigma$.

\

The transport of structures along bijections represents a functorial approach to combinatorics.
Within this approach, structures can be thought of as labeled by elements of the set $U$.
At the same time, in the case where the nature of these elements itself is not important, we can consider a structure $s\in F[U]$ and its image $F[\sigma](s)\in F[V]$ to be identical.
In particular, unless otherwise specified, we assume that the structures considered in this paper are labeled by the elements of the sets $[n]$,
where $[n]=\{1,\ldots,n\}$ for any positive integer $n$, and $[0]=\emptyset$.

\

\noindent
Given a species of structures $F$, a species $G$ is said to be a \emph{subspecies} of $F$ if
 \begin{enumerate}
  \item for each finite set $U$, we have $G[U]\subseteq F[U]$,
  \item for each bijection $\sigma\colon U\to V$, we have $G[\sigma]=F[\sigma]|_{G[U]}$.
 \end{enumerate}
In particular, for any species $F$, the following two subspecies are used quite frequently:
 \begin{itemize}
  \item the subspecies $F_+$ consisting of $F$-structures on nonempty sets,
  \item the subspecies $F_n$ consisting of $F$-structures on the set $[n]$, where $n\in\mathbb{Z}_{\geqslant0}$.
 \end{itemize}

\

For the purpose of enumeration, for a species of structure $F$, we use its \emph{exponential generating series}
 \[
  F(z) = \sum\limits_{n=0}^{\infty} \big|F[n]\big|\,\dfrac{z^n}{n!},
 \]
 where $F[n]$ is used as a short version of $F\big[[n]\big]$.

\

We also introduce the following notion of asymptotic probability.

\begin{notation}\label{notation: asymptotic probability} 
 Let $F$ be a species of structures.
 For a nonnegative integer $n$, in the case where the subspecies $F_n$ is nonempty, we endow it with the uniform probability $\mathbb{P}_n$: each object $s\in F[n]$ has probability $1/\big|F[n]\big|$.
 Now if $Q$ is some property, we denote by
 \[
  \mathbb{P}(s\mbox{ satisfies }Q)
 \]
 the sequence of probabilities
 \[
  \Big(\mathbb{P}_n\big\{s\in F[n]\mid s \mbox{ satisfies }Q\big\}\Big)_{|F[n]|>0}
 \]
 and call it the \emph{asymptotic probability that a random object $s\in F$ has the property $Q$}.
\end{notation}

\subsubsection{Particular examples of species}

The following species are of particular interest in the frame of this paper (for details, see~\cite{Joyal1981} and~\cite{BergeronLabelleLeroux1998}, as well as~\cite{Nurligareev2022}).
Here, we use the notation of~\cite{BergeronLabelleLeroux1998}, except for the species of permutations that we denote by $\mcP$ instead of $\mcS$.

\begin{itemize}
 \item
  The \emph{species of sets} $\mcE$, defined by
  \(
   \mcE[U] = \{U\}
  \)
  for every finite set $U$, and equipped with the trivial transport function.
  The corresponding exponential generating series is
  \[
   \mcE(z) = e^z.
  \]
 \item
  The species of \emph{linear orders} $\mcL$, defined by
  \[
   \mcL[U] = \{\phi_U\colon [n]\to U\mid\phi_U\mbox{ is bijective}\}
  \]
  for every finite set $U$ (for the appropriate integer $n$), and equipped with the transport function
  \begin{equation}\label{formula: transport for linear orders}
   \mcL[U]\to\mcL[V],
   \qquad
   \phi_U \mapsto \phi_V = \sigma\phi_U,
  \end{equation}
  along the bijection $\sigma\colon U\to V$ of finite sets $U$ and $V$.
  The exponential generating series of linear orders is
  \[
   \mcL(z) = \dfrac{1}{1-z}.
  \]
 \item
  The species of \emph{permutations} $\mcP$, defined by
  \[
   \mcP[U] = \{\psi_U\colon U\to U\mid\psi_U\mbox{ is bijective}\}
  \]
  for every finite set $U$, and equipped with the transport function
  \begin{equation}\label{formula: transport for permutations}
   \mcP[U]\to\mcP[V],
   \qquad
   \psi_U \mapsto \psi_V = \sigma\psi_U\sigma^{-1},
  \end{equation}
  along the bijection $\sigma\colon U\to V$ of finite sets $U$ and $V$.
  The exponential generating series $\mcP(z)$ coincides with that of linear orders:
  \[
   \mcP(z) = \dfrac{1}{1-z}.
  \]
 \item
  The species of \emph{cyclic permutations} $\mcCP$, defined as a subspecies of $\mcP$ that are cycles.
  Their exponential generating series is
  \[
   \mcCP(z) = \log\dfrac{1}{1-z}.
  \]
\end{itemize}

Note that $\mcZ = \mcE_1 = \mcL_1 = \mcP_1 = \mcCP_1$ is the species characteristic of singletons.
Other particular restrictions are $\One = \mcE_0 = \mcL_0 = \mcP_0$, which represents the species characteristic of the empty set, and $\Zero = \mcCP_0$, which is the empty species.
The corresponding exponential generating series are $\mcZ(z) = z$, as well as $\One(z) = 1$ and $\Zero(z)=0$, respectively.
More generally,
\[
 \mcE_n(z) = \dfrac{z^n}{n!}
 \qquad\mbox{and}\qquad
 \mcL_n(z) = \mcP_n(z) = z^n.
\]

\subsubsection{Equipotent and isomorphic species}

As we have seen, from the counting point of view the species of linear orders and permutations are similar:
the number of structures on $[n]$ is $n!$ in both cases, and their exponential generating series coincide, $\mcL(z)=\mcP(z)$.
For such type of species, it is said that they are \emph{equipotent}, and one writes $\mcL\equiv\mcP$.
In other terms, two species $F$ and~$G$ are \emph{equipotent} if their exponential generating series coincide.
In particular, this means that for each finite set $U$, there is a bijection $\alpha_U\colon F[U]\to G[U]$.

However, we cannot call species $\mcL$ and $\mcP$ equal, as the ways in which they behave are different.
Let us explain it in the following way.
Any two linear orders are \emph{isomorphic} in the sense that one can be obtained from another by a simple relabeling.
Formally, for any two linear orders $l_u\in \mcL[U]$ and $l_v\in \mcL[V]$ of the same size, there exists a bijection $\sigma\colon U\to V$ such that $l_v = \sigma \cdot l_u$, this is a consequence of~\eqref{formula: transport for linear orders}.
On the other hand, as follows from relation~\eqref{formula: transport for permutations}, two permutations are isomorphic if and only if they are of the same cycle type.
Thus, there exist permutations of the same size that are not isomorphic.

In the general case, we say that two species of structures are \emph{isomorphic} or \emph{equal}
if, for any finite set $U$, there is a bijection $\alpha_U\colon F[U] \to G[U]$ such that, for any $F$-structure $s\in F[U]$ and any bijection $\sigma\colon U\to V$, we have
\[
 \sigma\cdot\alpha_U(s) = \alpha_V(\sigma\cdot s).
\]
That is, the following diagram is commutative:
\[
 \begin{CD}
  F[U] @>\alpha_U>> G[U] \\
  @VF[\sigma]VV
  @VVG[\sigma]V \\
  F[V] @>\alpha_V>> G[V]
 \end{CD}
\]

\subsubsection{Operations on species}

Here, we briefly describe several operations that are useful for creating new species based on existing ones.
We omit the definition of transport functions because it naturally follows from that of sets of structures.
The reader can find all the necessary details in~\cite{BergeronLabelleLeroux1998}.

\

Let $F$ and $G$ be two species of structures.
\begin{itemize}
 \item
  An $(F+G)$-structure is either an $F$-structure or a $G$-structure.
  In other words, for each finite set $U$, the \emph{sum} $(F+G)$ produces a disjoint union
  \[
   (F+G)[U] = F[U] + G[U].
  \]
  Their exponential generating series satisfy
  \[
   (F+G)(z) = F(z) + G(z).
  \]
 \item
  An $(F \cdot G)$-structure is an ordered pair $(s,t)$ such that $s$ is an $F$-structure and $t$~is a $G$-structure.
  More precisely, for each finite set $U$, the \emph{product} $F\cdot G$ produces a~disjoint union of Cartesian products
  \[
   (F\cdot G)[U] = \sum_{\substack{U_F\cup U_G=U\\ U_F\cap U_G=\emptyset}} F[U_F] \times G[U_G].
  \]
  The corresponding exponential generating series satisfies
  \[
   (F \cdot G)(z) = F(z) \cdot G(z).
  \]
 \item
  Assuming that $G_0=\emptyset$,
  we define an $(F\circ G)$-structure as an $F$-assembly of disjoint $G$-structures.
  Formally, for each finite set $U$, the \emph{composition} $F\circ G$ produces a~disjoint union of Cartesian products 
  \[
   (F\circ G)[U] = \sum\limits_{\pi\in\Par[U]} F[\pi] \times \prod\limits_{p\in\pi}G[p],
  \]
  where $\Par[U]$ is the set of partitions of $U$.
  Here, we have the relation
  \[
   (F \circ G)(z) = F\big(G(z)\big).
  \]
 \item
  Finally, for each finite set $U$, the \emph{derivative} $F'$ produces the set
  \[
   F'[U] = F[U^{+}],
  \]
  where $U^{+} = U\cup\{*_U\}$ and $*_U\notin U$.
  In other words, an $F'$-structure is essentially an $F$-structure, but on a larger set enriched by an additional element.
  The exponential generating series of the corresponding species satisfies
  \[
   F'(z) = \dfrac{d}{dz}F(z).
  \]
\end{itemize}

The operations mentioned here possess natural properties that the reader would expect, such as $(F+G)\cdot H = F\cdot H + G\cdot H$ (distributive law), $(F\cdot G)' = F'\cdot G + F\cdot G'$ (Leibniz rule), $F\circ\mcZ = \mcZ\circ F = F$, etc.
The particular species discussed above are also related to each other by means of some of these operations.
Thus, we have
\[
 \mcP = \mcE\circ\mcCP
\]
and
\[
 \mcL_n=\mcZ^n.
\]
We will also use the following relations concerning the derivatives of the species of sets, linear orders, and cyclic permutations, respectively:
\[
 \mcE'=\mcE,
 \qquad
 \mcL'=\mcL^2,
 \qquad
 \mcCP'=\mcL,
\]
as well as the derivatives of their restrictions to a given cardinality:
\[
 (\mcE_n)'=\mcE_{n-1},
 \qquad
 (\mcL_n)'=n\mcL_{n-1},
 \qquad
 (\mcCP_n)'=\mcL_{n-1}.
\]

\subsubsection{Virtual species}

The class of species of structures forms a semi-ring with respect to addition ($+$) and multiplication ($\cdot$), where zero and one are the species $\Zero$ and $\One$.
That is why it can be extended in a similar manner as the semi-ring $\mathbb{Z}_{\geqslant0}$ of natural numbers is extended to build the ring $\mathbb{Z}$ of integers.
The result of this extension is known as the class of \emph{virtual species}, and is shown to be compatible with other operations, such as composition ($\circ$) and derivation ($'$).

In addition to the subtraction operation, virtual species have inverses for multiplication and substitution.
\begin{itemize}
 \item
  If a (virtual) species $F$ satisfies $F_0=\One$, then, in the class of virtual species, it admits the unique \emph{multiplicative inverse}:
  \[
   F^{-1} = \sum\limits_{k=0}^{\infty}(-1)^k(F_+)^k.
  \]
  In the particular case where $F=\mcL$, this relation can be simplified:
  \[
   \mcL^{-1}=\One-\mcZ.
  \]
  As expected, the exponential generating series of the multiplicative inverse satisfies $F^{-1}(z)=1/F(z)$.
 \item 
  If a (virtual) species $F$ satisfies $F_0=\Zero$ and $F_1=\mcZ$, then, in the class of virtual species, it admits unique \emph{inverse for substitution}:
  \[
   F^{(-1)} = \sum\limits_{k=0}^{\infty}(-1)^k\Delta_{F}^k(\mcZ),
  \]
 where the linear operator $\Delta_F$ is defined by the relation
 \(
  \, \Delta_F(\Phi) = \Phi\circ F - \Phi.
 \)
 Its exponential generating series satisfies
 \(
  \, F\big(F^{(-1)}(z)\big) = F^{(-1)}\big(F(z)\big) = z.
 \)
\end{itemize}
Again, these inverses behave exactly as the reader would naturally expect.
For instance, taking a~derivative of the identity
 \[
  \Psi^{(-1)} \circ \Psi = \mcZ
 \]
 leads us to the expression for the derivative of the inverse under substitution:
 \[
  (\Psi^{(-1)})' = (\Psi')^{-1}\circ\Psi^{(-1)}.
 \]
 For more detailed information on virtual species, we invite the interested reader to refer to~\cite[Section 2.5]{BergeronLabelleLeroux1998}.

\subsubsection{Weighted species}

It is quite common in enumerative combinatorics to count certain objects according to specific parameters.
To this end, marking variables are often used, and weights are assigned to the objects under consideration.
Applying this approach to species theory, we arrive at the concept of weighted species.

Formally, our set of all possible weights is a ring of formal power series $\mathbb{A}$ (in an arbitrary number of variables) with coefficients in an integral domain ($\mathbb{Z}$ or $\mathbb{R}$ in our case).
We extend the definition of species discussed above to \emph{$\mathbb{A}$-weighted species} as follows.
The set $F[U]$, which is produced for a finite set $U$, is now $\mathbb{A}$-weighted, finite or summable.
The latter means that $F[U]$ is equipped with a weight function $w\colon F[U]\to\mathbb{A}$ preserved by the transport function, and that there exists the sum
\[
 \big|F[U]\big|_w = \sum\limits_{s\in F[U]}w(s).
\]
called the \emph{inventory} of the set of $F$-structures on $U$.
The value $\big|F[n]\big|_w$ will be referred to as the \emph{total weight} of $F$-structures on $[n]$.

All of the above concepts can be naturally extended to the weighted case.
Thus, the exponential generating series of an $\mathbb{A}$-weighted species is now
\[
 F(z) = \sum\limits_{n=0}^{\infty}\big|F[n]\big|_w\dfrac{z^n}{n!}.
\]
The weight of a product of two species is defined as the product of their weights,
the weight of a composition is the product of the weights of its components, etc.
The concept of virtual species matches the weights as well.
As usual, for all details, we refer the reader to the book~\cite{BergeronLabelleLeroux1998}.

Note that, in the case where all weights are positive, it is natural to endow $F[n]$ with a discrete probability:
a structure $s\in F[n]$ appears with probability
 \[
  \mathbb{P}(s) = \dfrac{w(s)}{\big|F[n]\big|_w}.
 \]

\subsection{Gargantuan sequences and Bender's theorem}
\label{sec: gargantuan}

In this section, we introduce the concepts of gargantuan sequence and gargantuan species, which first appeared in~\cite{MonteilNurligareevSET} and~\cite[Chapter~2]{Nurligareev2022}.
We also recall a simplified version of Bender's theorem~\cite{Bender1975} (see also~\cite{Odlyzko1995}),
which serves as the main asymptotic tool in this work.

\begin{definition}\label{def: gargantuan species}
 A sequence $(a_n)$ is \emph{gargantuan} if the following two conditions hold:
 \[
 \mbox{ (i) }\quad
 \dfrac{a_{n-1}}{a_n}\to0,
 \quad\mbox{as }n\to\infty;
 \qquad\quad
 \mbox{ (ii) }\quad
 \sum\limits_{k=r}^{n-r}|a_ka_{n-k}| = O\big(a_{n-r}\big)\,\,\,\mbox{ for each }\, r\in\bbZ_{\geqslant0}.
 \]
 Furthermore, let $F$ be a (virtual, weighted) species of structures.
 We call $F$ \emph{gargantuan} if the sequence $a_n = \big(\big|F[n]\big|/n!\big)$ is gargantuan.
 In particular, we assume that $a_n\neq0$ for almost all $n\in\bbZ_{\geqslant0}$.
\end{definition}

\begin{theorem}\label{theorem: Bender}
 Consider a formal power series
 \[
  U(z) = \sum\limits_{n=1}^{\infty}u_nz^n
 \]
 and a function $F(x)$ analytic in a neighborhood of origin.
 Define
 \[
  V(z) = \sum\limits_{n=0}^{\infty}v_nz^n = F\big(U(z)\big)
 \qquad
 \mbox{and}
 \qquad
  W(z) = \sum\limits_{n=0}^{\infty}w_n z^n =
  \left.\frac{\partial}{\partial x} F(x)\right|_{x=U(z)}.
 \]
 Assume that $u_n\ne 0$ for all sufficiently large $n$, and that the sequence $(u_n)$ is gargantuan.
 In this case,
 \[
  v_n \approx \sum\limits_{k\ge0}w_ku_{n-k}
 \]
 and the sequence $(v_n)$ is gargantuan.
\end{theorem}

In order to verify the gargantuan property in applications, we rely on the following sufficient condition.

\begin{lemma}[Lemma~2.4 in~\cite{MonteilNurligareevSET}]
\label{lemma: sufficient conditions for gargantuan sequence}
 If a sequence $(a_n)$ satisfies the following two conditions,
 \begin{align*}\label{equation: sufficient conditions for gargantuan sequence}
  \emph{ (i)' }\quad &
  na_{n-1} = O(a_n),\,\,\mbox{ as }\, n\to\infty, \\
  \emph{ (ii)' }\quad &
  x_k = |a_ka_{n-k}| \mbox{ is decreasing for } k < n/2 \mbox{ and for all but finitely many } n,
 \end{align*}
 then $(a_n)$ is gargantuan.
\end{lemma}

The property of being gargantuan is closed with respect to the point-wise product and to point-wise multiplication with a geometric progression.

\begin{lemma}[Lemma~2.5 in~\cite{MonteilNurligareevSET}]
\label{lemma: a_nb_n is gargantuan}
 If two sequences $(a_n)$ and $(b_n)$ are gargantuan, then the sequence $(a_nb_n)$ is also gargantuan.
\end{lemma}

\begin{lemma}\label{lemma: a_nс^n is gargantuan}
 If a sequence $(a_n)$ is gargantuan and $b_n=b_0c^n$ with $c\neq0$, then the sequence $(a_nb_n)$ is gargantuan.
\end{lemma}
\begin{proof}
 The first condition of Definition~\ref{def: gargantuan species} holds, since
 \[
  \dfrac{a_{n-1}b_{n-1}}{a_nb_n} = \dfrac{1}{c} \cdot \dfrac{a_{n-1}}{a_n}.
 \]
 As for the second condition, it holds as well, because
 \[
  \sum\limits_{k=r}^{n-r} \big|(a_kb_k)\cdot(a_{n-k}b_{n-k})\big|
   =
  b_0^2|c^n|\cdot\sum\limits_{k=r}^{n-r}|a_ka_{n-k}|
   =
  |b_rb_{n-r}|\cdot O(a_{n-r})
   =
  O(a_{n-r}b_{n-r}).
 \]
\end{proof}

\section{Asymptotic theorems and ``anti-SEQ'' operator}
\label{sec: main result}

This section is devoted to our general asymptotic result.
Given a species of structures~$\mcA$ that can be represented as a composite $\mcA=\mcE\circ\mcB$, our objective is to describe the asymptotic behavior of its subspecies $\mcE_m\circ\mcB$ for any positive integer $m$.
From a structural point of view, the species of structures $\mcB$ are interpreted as connected components of the species of structures $\mcA$.
Thus, $\mcE_m\circ\mcB$ represents subspecies of $\mcA$ with exactly $m$ connected components.
For simplicity, we denote these subspecies by $\mcB^{\{m\}}$.

We begin by establishing the asymptotic probability that a random $\mcA$-structure on~$[n]$ belongs to $\mcB^{\{m\}}$ (Section~\ref{subsec: asymptotic theorem for E}).
In particular, in the case where $m=1$, this gives us the probability that an $\mcA$-structure is connected in terms of the ``anti-$\SEQ$'' operator.
We then combinatorially interpret the coefficients appearing in these complete asymptotic expansions (Section~\ref{subsec: asymptotic interpretation}).
Finally, we generalize the obtained result to different settings.
More precisely, in Section~\ref{subsec: asymptotic theorem for L and C} we consider species of the form $\mcA = \mcL\circ\mcB$ and $\mcA = \mcCP\circ\mcB$, while Section~\ref{subsec: p-periodic sequences} is devoted to the case where the sequence $(\a_n)$ of total weights of the species $\mcA$ on the sets $[n]$ is $p$-periodic for some $p>1$.

\subsection{Asymptotic theorem}
\label{subsec: asymptotic theorem for E}

\begin{theorem}\label{theorem: species E asymptotics}
 Let $\mcA$ be a gargantuan (weighted) species of structures that can be represented as $\mcA=\mcE\circ\mcB$ for some (weighted) species of structures~$\mcB$.
 Assume that $m,n\in\bbZ_{\geqslant0}$, and $s\in\mcA$ is a random $\mcA$-structure on $[n]$.
 In this case,
 \begin{equation}\label{formula: species E asymptotics}
  \bbP(s\in\mcB^{\{m\}})
   \approx 
  \sum\limits_{k\ge0}\d_{k,m}\cdot\binom{n}{k}\cdot\dfrac{\a_{n-k}}{\a_n},
 \end{equation}
 where $\a_n$ and $\d_{n,m}$ are the total weights on the set $[n]$ of the species of structures $\mcA$ and $\mcD(m) = \mcB^{\{m-1\}}(\mcE^{-1}\circ\mcB)$, respectively.
\end{theorem}
\begin{proof}
 The main idea of the proof is to apply Theorem~\ref{theorem: Bender}.
 First, we express the species $\mcB^{\{m\}} = \mcE_m\circ\mcB$ via $\mcA_+ = \mcA - \One$:
 \[
  \mcA_+ = (\mcE_++\One)\circ\mcB - \One = \mcE_+\circ\mcB,
 \]
 and hence,
 \[
  \mcB = \mcE_+^{(-1)}\circ\mcA_+
  \qquad\mbox{and}\qquad
  \mcB^{\{m\}} = \mcE_m\circ\mcE_+^{(-1)}\circ\mcA_+.
 \]
 Second, we observe that the exponential generating series 
 \[
  \left(\mcE_m\circ\mcE_+^{(-1)}\right)(z) = \dfrac{\big(\log(1+z)\big)^m}{m!}
 \]
 is analytic in some neighborhood of origin.
 Therefore, Theorem~\ref{theorem: Bender} is applicable to the exponential generating series of $\mcU = \mcA_+$ and $F = \mcE_m\circ\mcE_+^{(-1)}$ taken as the formal power series $U(z)$ and the analytic function $F(x)$, respectively.
 Taking into account that $F\big(U(z)\big)$ is the exponential generating series of the species $\mcV=\mcB^{\{m\}}$, this gives
 \begin{equation}\label{formula: asymptotics proof}
  \b_n^{\{m\}} \approx \sum\limits_{k\ge0}\binom{n}{k}\d_{k,m}\a_{n-k},
 \end{equation}
 where $\b_k^{\{m\}}$ and $\d_{k,m}$ are the total weights on the set $[k]$ of the species of structures $\mcB^{\{m\}}$ and $\mcD(m) = F'\circ\mcA_+$, respectively.
 Since
 \[
  F' = \big(\mcE_m\circ\mcE_+^{(-1)}\big)' =
  \Big(\mcE_m'\circ\mcE_+^{(-1)}\Big)
   \cdot \Big(\big((\mcE_+)'\big)^{-1}\circ\mcE_+^{(-1)}\Big),
 \]
 and taking into account that $\mcE'_m=\mcE_{m-1}$ and $(\mcE_+)'=\mcE$, we have
 \[
  \mcD(m) = \big(\mcE_{m-1}\cdot\mcE^{-1}\big)\circ\mcB =
  \mcB^{\{m-1\}}\cdot(\mcE^{-1}\circ\mcB). 
 \]
 Now, using the relation $\bbP(s\in\mcB)=\b_n^{\{m\}}/\a_n$, we divide both sides of formula~\eqref{formula: asymptotics proof} by $\a_n$ to complete the proof.
\end{proof}

\begin{corollary}\label{cor: leading term of species E asymptotics}
 If $\a_1\ne0$, then the leading term of asymptotic expansion~\eqref{formula: species E asymptotics} satisfies
 \begin{equation}\label{formula: leading term of species E asymptotics}
  \bbP(s\in\mcB^{\{m\}}) 
   =
  \binom{n}{m-1}\cdot\dfrac{\a_1^{m-1}\a_{n-m+1}}{\a_n}
   +
  O\left(n^m\cdot\dfrac{\a_{n-m}}{\a_n}\right),
 \end{equation}
 where $(n)_{m-1} = n(n-1)\ldots(n-m+2)$ are the falling factorials.
\end{corollary}
\begin{proof}
 Given a fixed positive integer $m$, the leading term of~\eqref{formula: species E asymptotics} corresponds to the first nonzero element of the sequence $(\d_{k,m})$.
 The behavior of this element is determined by the generating series of the species $\mcD(m) = \mcB^{\{m-1\}}\cdot(\mcE^{-1}\circ\mcB)$.
 Relation~\eqref{formula: leading term of species E asymptotics} now follows from the fact that
 \[
  \mcB^{\{m-1\}}(z)
   = 
  \dfrac{(\a_1z + \ldots)^{m-1}}{(m-1)!}
   =
  \a_1^{m-1}\cdot\dfrac{z^{m-1}}{(m-1)!} + O(z^m),
 \]
 and that, consequently, the first nonzero element is $\d_{m-1,m} = \a_1^{m-1}$.
\end{proof}

\begin{lemma}\label{lemma: magical species identity}
 The virtual species $\mcE^{-1}\circ\mcE_+^{(-1)}$ and $\big(\One-\mcL_+^{(-1)}\big)$ are equipotent.
\end{lemma}
\begin{proof}
 Straightforward calculations show that
 \[
  \mcE^{-1}(z) = e^{-z},\qquad
  \mcE_+^{(-1)}(z) = \log(1+z),\qquad
  \mcL_+^{(-1)}(z) = \dfrac{z}{1+z},
 \]
 and hence,
 \[
  \mcE^{-1}\big(\mcE_+^{(-1)}(z)\big) + \mcL_+^{(-1)}(z) = \dfrac{1}{1+z} + \dfrac{z}{1+z} = 1.
 \]
\end{proof}

Due to Lemma~\ref{lemma: magical species identity}, we can interpret the coefficients appearing in Theorem~\ref{theorem: species E asymptotics} in terms of the species $\mcL_+^{(-1)}$ understood as an ``anti-$\SEQ$'' operator.
Indeed, we have
 \[
  \mcL_+^{(-1)}
   \equiv 
  \One - \mcE^{-1}\circ\mcE_+^{(-1)}.
 \]
Hence, in view of Theorem~\ref{theorem: species E asymptotics},
if a gargantuan species of structures $\mcA$ satisfies $\mcA = \mcE \circ \mcB$,
then the asymptotic behavior of $\mcB$ is encoded by the species of structures
 \[
  \mcD(1) = \mcE^{-1}\circ\mcB 
  \equiv
  \big(\One -\mcL_+^{(-1)}\big)\circ\mcA_+ = \One - \mcL_+^{(-1)}\circ\mcA_+,
 \]
In other words, the sequences of $\One-\mcD(1)$ are equipotent to $\mcA$.

\subsection{Interpretation of the asymptotic coefficients}
\label{subsec: asymptotic interpretation}

The constants $\d_{k,m}$ appearing in the statement of Theorem~\ref{theorem: species E asymptotics} possess a combinatorial interpretation in terms of the species of structures $\mcB$ that serves as a building block.
However, the interpretation is not straightforward.
To establish its explicit form, we first need to interpret the species of structures $\mcE^{-1}$.

\begin{lemma}\label{lemma: E^(-1) = sum((-1)^(k)E_k)}
 The virtual species of structures $\mcE^{-1}$ is equipotent to the alternating sum
 \begin{equation}\label{formula:tilde(E)}
  \tilde{\mcE} = \One - \mcE_1 + \mcE_2 - \mcE_3 + \ldots \,,
 \end{equation}
 \emph{i.e.} their exponential generating series coincide.
\end{lemma}
\begin{proof}
 The corresponding exponential generating series both are equal to
 \[
  \dfrac{1}{e^z} = e^{-z} =
  1 - z + \dfrac{z^2}{2!} - \dfrac{z^3}{3!} + \ldots
 \] 
\end{proof}
\begin{remark}\label{remark: E^(-1) = sum((-1)^(k)E_k)}
 The species $\mcE^{-1}$ and $\tilde{\mcE}$ indicated in Lemma~\ref{lemma: E^(-1) = sum((-1)^(k)E_k)} are not equal.
 To see this, it suffices to consider their restrictions on cardinality $2$.
 Here, we have $(\mcE^{-1})_1=\mcE_1^2-\mcE_2$ and $\tilde{\mcE}_1=\mcE_2$, but $\mcE_1^2\neq2\mcE_2$.
\end{remark}

Lemma~\ref{lemma: E^(-1) = sum((-1)^(k)E_k)} allows us to think of the coefficients of the species $\mcD(1)=\mcE^{-1}\circ\mcB$ as the coefficients of the alternating sum
 \[
  \tilde{\mcE} \circ \mcB =
  \One - \mcE_1\circ\mcB + \mcE_2\circ\mcB - \mcE_3\circ\mcB + \ldots = 
  \One - \mcB^{\{1\}} + \mcB^{\{2\}} - \mcB^{\{3\}} + \ldots
 \]
The latter relation can be interpreted similarly to
 \[
  \mcA = \mcE \circ \mcB =
  \One + \mcB^{\{1\}} + \mcB^{\{2\}} + \mcB^{\{3\}} + \ldots,
 \]
where $\mcB$-structures are thought of as connected $\mcA$-structures.
Indeed, the species $\tilde{\mcE}\circ\mcB$ consists of the same structures as $\mcA$, but a structure is considered negative if the number of its connected components is odd.
In the case where the species $\mcA$ and $\mcB$ are not weighted, this provides a combinatorial meaning for the coefficients of $\mcD(1)$ in terms of two counting sequences, which we express by the following proposition.

\begin{proposition}\label{prop: d_(k,1) interpretation}
 Let $\mcA$ be a gargantuan species of structures that can be represented as $\mcA=\mcE\circ\mcB$ for some species of structures~$\mcB$.
 Assume that $n$ is a positive integer and $s\in\mcA$ is a random $\mcA$-structure on $[n]$.
 In this case,
 \begin{equation}\label{formula: species E_1 asymptotics}
  \bbP(s\mbox{ is connected})
   \approx 
  1 - \sum\limits_{k\ge1}\d_k\cdot\binom{n}{k}\cdot\dfrac{\a_{n-k}}{\a_n},
 \end{equation}
 where $\a_n$ is the counting sequence of $\mcA$,
 \begin{equation}\label{formula: d_(k,1) interpretation}
  \d_{k}
   = 
  \#\{s'\in\mcA_k \mid \pi_0(s') \mbox{ is odd}\}
   -
  \#\{s'\in\mcA_k \mid \pi_0(s') \mbox{ is even}\}
\end{equation}
 and $\pi_0(s')$ is the number of connected components of the structure $s'$.
\end{proposition}

More generally, the species $\mcB^{\{m-1\}}\cdot(\mcE^{-1}\circ\mcB)$ is equipotent to the alternate sum
 \[
  \mcB^{\{m-1\}} - \mcB^{\{m-1\}}\cdot\mcB^{\{1\}} + \mcB^{\{m-1\}}\cdot\mcB^{\{2\}} - \mcB^{\{m-1\}}\cdot\mcB^{\{3\}} + \ldots
 \]
For non-weighted species, this leads to the following interpretation of the coefficients $\d_{k,m}$.

\begin{proposition}\label{prop: d_(k,m) interpretation}
 Let $\mcA$ be a gargantuan species of structures that can be represented as $\mcA=\mcE\circ\mcB$ for some species of structures~$\mcB$.
 Assume that $n$ is a positive integer and $s\in\mcA$ is a random $\mcA$-structure on $[n]$.
 In this case,
 \begin{equation}\label{formula: species E_m asymptotics}
  \bbP(s\mbox{ has }m\mbox{ connected components})
   \approx 
  \sum\limits_{k\ge0}\d_{k,m}\cdot\binom{n}{k}\cdot\dfrac{\a_{n-k}}{\a_n},
 \end{equation}
 where $\a_n$ is the counting sequence of $\mcA$,
 \begin{align*}
  \d_{k,m}
   & = 
  \#\big\{(s',s'')\in\mcA_i\cdot\mcA_{k-i} \mid 0\leqslant i\leqslant k,\,\, \pi_0(s')=m-1,\,\, \pi_0(s'') \mbox{ is even}\big\} \\
   & -
  \#\big\{(s',s'')\in\mcA_i\cdot\mcA_{k-i} \mid 0\leqslant i\leqslant k,\,\, \pi_0(s')=m-1,\,\, \pi_0(s'') \mbox{ is odd}\big\}
 \end{align*}
 and $\pi_0(s')$ is the number of connected components of the structure $s'$.
\end{proposition}

In particular, from Proposition~\ref{prop: d_(k,m) interpretation} it is seen that the dominant term of asymptotics~\eqref{formula: species E_m asymptotics} is determined by $\mcB^{\{m-1\}}$-structures of minimal size,
that is, $\mcA$-structures with $(m-1)$ connected components whose size is minimal (compare with Corollary~\ref{cor: leading term of species E asymptotics}).

\subsection{Adaptation for linear orders and cycles}
\label{subsec: asymptotic theorem for L and C}

\begin{theorem}\label{theorem: species L asymptotics}
 Let $\mcF\in\{\mcL,\mcCP\}$, and let $\mcA$ be a gargantuan (weighted) species of structures that can be represented as $\mcA=\mcF\circ\mcB$ for some (weighted) species of structures~$\mcB$.
 Assume that $m,n\in\bbZ_{\geqslant0}$, and $s\in\mcA$ is a random $\mcA$-structure on $[n]$.
 In this case,
 \begin{equation}\label{formula: species L asymptotics}
  \bbP(s\in\mcF_m\circ\mcB)
   \approx 
  \sum\limits_{k\ge0}\d_{k,m}(\mcF)\cdot\binom{n}{k}\cdot\dfrac{\a_{n-k}}{\a_n},
 \end{equation}
 where $\a_n$, $\d_{n,m}(\mcL)$, and $\d_{n,m}(\mcC)$ are the total weights on the set $[n]$ of the species of structures $\mcA$, $m\mcB^{m-1}(1-\mcB)^2$, and $\mcB^{m-1}(1-\mcB)$, respectively.
\end{theorem}
\begin{proof}
 The idea of the proof is essentially the same as in Theorem~\ref{theorem: species E asymptotics}: we apply Theorem~\ref{theorem: Bender} to the exponential generating series of the species $\mcU = \mcA_+$ and another species~$F$ taken so that $F(U(z))$ is the exponential generating series of the species $\mcV=\mcF_m\circ\mcB$.
 The key steps, along with the comparison with the corresponding steps of the proof of Theorem~\ref{theorem: species E asymptotics}, are reflected in Table~\ref{table: species asymptotics proof}.
 For further details, we refer the reader to~\cite{Nurligareev2022}.
 \begin{table}[ht!]
  \centering
  \begin{tabular}{|c||c|c|c|c|}
   \hline
    $\mcF$ & $\mcA_0$ & $F$ & $F(x)$ & $F'$ \\
   \hline
    $\mcE$ & $\One$ & $\mcE_m\circ\mcE_+^{-1}$ & $\dfrac{\log^m(1+x)}{m!}$ & $\big(\mcE_{m-1}\circ\mcE_+^{(-1)}\big)\cdot\big(\mcE^{-1}\circ\mcE_+^{(-1)}\big)$ \\
    $\mcL$ & $\One$ & $\mcL_m\circ\mcL_+^{-1}$ & $\left(\dfrac{x}{1+x}\right)^m$ & $\big(m\mcL_{m-1}\circ\mcL_+^{(-1)}\big)\cdot\big(\mcL^{-2}\circ\mcL_+^{(-1)}\big)$ \\ 
    $\mcCP$ & $\Zero$ & $\mcCP_m\circ\mcCP^{(-1)}$ & $\dfrac{(1-e^{x})^m}{m}$ & $\big(\mcL_{m-1}\circ\mcCP^{(-1)}\big)\cdot\big(\mcL^{-1}\circ\mcCP^{(-1)}\big)$ \\
   \hline
  \end{tabular}
  \caption{Proof key steps summary.}
  \label{table: species asymptotics proof}
 \end{table}
\end{proof}

\begin{corollary}\label{cor: leading term of species L asymptotics}
 If $\a_1\ne0$, then the leading term of asymptotic expansion~\eqref{formula: species L asymptotics} satisfies
 \begin{equation}\label{formula: leading term of species asymptotics}
  \bbP(s\in\mcF_m\circ\mcB) 
   =
  c(\mcF)\cdot(n)_{m-1}\cdot\dfrac{\a_1^{m-1}\a_{n-m+1}}{\a_n}
   +
  O\left(n^m\cdot\dfrac{\a_{n-m}}{\a_n}\right),
 \end{equation}
 where $(n)_{m-1} = n(n-1)\ldots(n-m+2)$ are the falling factorials, $c(\mcL)=m$, and $c(\mcCP)=1$.
\end{corollary}

\begin{remark}\label{rem: asymptotic theorem for virtual species}
 Theorems~\ref{theorem: species E asymptotics} and~\ref{theorem: species L asymptotics} can be applied even if the species $\mcA$ is virtual.
 In this case, it is no longer possible to discuss probabilities,
 so the left side of asymptotic relations~\eqref{formula: species E asymptotics} and~\eqref{formula: species L asymptotics} should be replaced by the quotient $\b_n(\mcF)/\a_n$,
 where $\a_n$ and $\b_n(\mcF)$ are the total weights on the set $[n]$ of, respectively, the species $\mcA$ and $\mcF_m\circ\mcB$ with $\mcF\in\{\mcE,\mcL,\mcCP\}$.
\end{remark}

\subsection{Adaptation for p-periodic sequences}
\label{subsec: p-periodic sequences}

In the settings of Theorems~\ref{theorem: species E asymptotics} and~\ref{theorem: species L asymptotics}, the sequence $(a_n)$ of the total weights of the species~$\mcA$ is allowed to have finitely many zeros.
In practice, however, it happens that $(a_n)$ contains an infinite number of zeros, as, for example, in the case of perfect matchings or triangulated surfaces.
Formally, the above theorems cannot be applied in these circumstances.
Nevertheless, the underlying asymptotic approach can still be used due to the regularity of the zero distribution guaranteed by the other theorem conditions.
More precisely, it can be shown~\cite{Wright1967} that if two species of structures $\mcA$ and $\mcB$ satisfy $\mcA=\mcF \circ \mcB$ where $\mcF\in\{\mcE, \mcL, \mcCP\}$, 
then the sequence $(\a_n)$ is \emph{$p$-periodic} for some positive integer $p$, \emph{i.e.} 
\begin{itemize}
 \item
  $\a_n\neq0$ for $n=pk$, where $k$ is sufficiently large,
 \item
  $\a_n=0$ in any other case.
\end{itemize}
Thus, for species whose sequences of total weight are $p$-periodic, Theorems~\ref{theorem: species E asymptotics} and~\ref{theorem: species L asymptotics} can be adapted in the following way.

\begin{proposition}\label{prop: p-periodic species asymptotics}
 Let $\mcA$, $\mcB$, and $\mcF$ be three (weighted) species of structures that satisfy $\mcA = \mcF \circ \mcB$.
 In addition, let the sequence $(\a_n)$ of the total weights of the species of structures $\mcA$ be $p$-periodic for some positive integer $p$,
 and let the sequence $\big(\a_{pn}/(pn)!\big)$ be gargantuan.
 Suppose that $m,n\in\mathbb{Z}_{\geqslant0}$ and $s\in\mcA$ is a random $\mcA$-structure on $[n]$.
 In this case,
 \begin{equation}\label{formula: p-periodic species asymptotics}
  \bbP(s\in\mcF_m\circ\mcB)
   \approx 
  \sum\limits_{k\ge0}\d_{pk,m}\cdot\binom{pn}{pk}\cdot\dfrac{\a_{p(n-k)}}{\a_{pn}},
 \end{equation}
 where the species of structures $\mcD = \mcD(m)$ depends on $\mcF$ according to Table~\ref{table: species asymptotics},
 and the numbers $\d_{n,m}$ are the total weights on $[n]$ of the species $\mcD$.
 \begin{table}[ht!]
  \centering
  \begin{tabular}{|c||c|c|c|}
   \hline
    $\mcF$ & $\mcF_m$ & $\mcD(m)$ & $\mcD(1)$ \\
   \hline
    $\mcE$ & $\mcE_m$ & $\mcB^{\{m-1\}}(\mcE^{-1}\circ\mcB)$ & $\mcE^{-1}\circ\mcB$ \\
    $\mcL$ & $\mcL_m$ & $m\mcB^{m-1}(\One-\mcB)^2$ & $(\One-\mcB)^2$ \\ 
    $\mcCP$ & $\mcCP_m$ & $\mcB^{m-1}(\One-\mcB)$ & $\One-\mcB$ \\
   \hline
  \end{tabular}
  \caption{Correspondence between $\mcF$ and $\mcD$.}
  \label{table: species asymptotics}
 \end{table}
\end{proposition}
\begin{proof}
 Here, we apply Theorem~\ref{theorem: Bender} to the formal power series
 \[
  U(z) = \mcA_+(z^{1/p})
  =
  \sum\limits_{n=0}^{\infty} \a_{np} \dfrac{z^n}{(np)!}
 \]
 and the analytic function $F(x)$ indicated in Table~\ref{table: species asymptotics proof}.
 Similarly to the proof of Theorem~\ref{theorem: species E asymptotics}, we ensure that the conditions of Theorem~\ref{theorem: Bender} hold and the same reasoning provides the desired result.
\end{proof}

\begin{corollary}\label{cor: leading term of p-periodic species asymptotics}
 If $\a_p\ne0$, then the leading term of asymptotic expansion~\eqref{formula: p-periodic species asymptotics} satisfies
 \begin{equation}\label{formula: leading term of p-periodic species asymptotics}
  \bbP(s\in\mcF_m\circ\mcB) 
   =
  c(\mcF) \cdot \dfrac{(pn)_{p(m-1)}}{(p!)^{m-1}} \cdot \dfrac{\a_p^{m-1}\a_{p(n-m+1)}}{\a_{pn}}
   +
  O\left(n^{pm}\cdot\dfrac{\a_{p(n-m)}}{\a_{pn}}\right).
 \end{equation}
 where $(pn)_{p(m-1)} = pn(pn-1)\ldots\big(p(n-m+1)+1\big)$ are the falling factorials.
 In particular, for $\mcF\in\{\mcE,\mcL,\mcCP\}$ the constants are $c(\mcE)=1/(m-1)!$, $c(\mcL)=m$, and $c(\mcCP)=1$.
\end{corollary}
\begin{proof}
 Since the sequence $(\a_n)$ is $p$-periodic, we have
 \[
  \mcB^{m-1}(z)
   = 
  \left(\a_p\dfrac{z^p}{p!} + \ldots\right)^{m-1}
   =
  \dfrac{\a_p^{m-1}}{(p!)^{m-1}} \cdot \big(p(m-1)\big)! \cdot \dfrac{z^{p(m-1)}}{\big(p(m-1)\big)!} + O(z^{pm}).
 \]
 Therefore, the first nonzero element of the sequence $(\d_{pk,m})$ corresponds to $k=m-1$ and equals
 \[
  \d_{p(m-1),m} = c(\mcF) \cdot \dfrac{\a_p^{m-1}}{(p!)^{m-1}} \cdot \big(p(m-1)\big)!.
 \]
 To complete the proof, it suffices to note that
 \[
  \d_{p(m-1),m}\cdot\binom{pn}{p(m-1)}\cdot\dfrac{\a_{p(n-m+1)}}{\a_{pn}}
   =
  c(\mcF) \cdot \dfrac{(pn)_{p(m-1)}}{(p!)^{m-1}} \cdot \dfrac{\a_p^{m-1}\a_{p(n-m+1)}}{\a_{pn}}.
 \]
\end{proof}

\section{Applications}
\label{sec: applications}

Theorems~\ref{theorem: species E asymptotics}, \ref{theorem: species L asymptotics} and Propositions~\ref{prop: d_(k,1) interpretation},~\ref{prop: d_(k,m) interpretation}~\ref{prop: p-periodic species asymptotics} generalize the asymptotic results obtained for labeled combinatorial classes in~\cite{MonteilNurligareevSET,MonteilNurligareevSEQ}.
In this section, we focus on the applications that cannot be treated using combinatorial class approach.
The first of them concerns the Erd\H{o}s-R\'enyi graph model, while the second refers to various models of surfaces and manifolds.

\mathversion{bold}
\subsection{The Erd\H{o}s-R\'enyi $G(n,p)$ model}
\mathversion{normal}
\label{subsec: Erdos-Renyi}

Let $p\in(0,1)$ be a fixed real number, and let $n\in\bbZ_{\geqslant0}$.
Here we consider the \emph{Erd\H{o}s--Rényi model} $G(n,p)$ of a random labeled graph:
its set of vertices is $[n]$, and each pair of vertices is joined by an edge independently with probability $p$.
Thus, a particular graph with $n$~vertices and $k$~edges is picked with probability
 \[
  p^{k}q^{\binom{n}{2}-k},
 \]
where $q=1-p$.

Take a random graph $g$ in the Erd\H{o}s--Rényi model $G(n,p)$.
Our goal is to provide and interpret the asymptotic probability that this graph is connected as $p$ is fixed and $n$~tends to infinity.
To this end, we introduce a weighted species of graphs compatible with the model and employ our asymptotic results described in Section~\ref{sec: main result}.

\subsubsection{Weighted species of graphs}

Let, as before, $p\in(0,1)$, and $q=1-p$.
We also define an additional parameter $\rho$ as their quotient,
 \[
  \rho = \dfrac{p}{q} = \dfrac{1}{q} - 1,
 \]
and consider the \emph{weighted species $\mcG$ of graphs}, that is, the species $\mcG$ of graphs with the weight assigned to a graph $g$ to be
 \[
  w(g) = \rho^{|E(g)|},
 \]
where $|E(g)|$ is the number of edges in the graph $g$.
For example, the weights of all graphs of size at most 3 are indicated in Fig.~\ref{figure: small graph weights}.

\begin{figure}[!ht]
\begin{center}
\begin{tikzpicture}[>=triangle 45, line width=.5pt]
 \begin{scope}[xshift=-6.5cm,yshift=1.2cm]
  \coordinate [label=90:$1$] (a1) at (0,0);
  \draw (0,-15pt) node {$w=1$};
  \filldraw [black] (a1) circle (1.5pt);
 \end{scope}
 \begin{scope}[xshift=-3.5cm]
  \coordinate (a1) at (0,10pt);
  \coordinate [label=180:$1$] (a2) at ([rotate = 120] a1);
  \coordinate [label=0:$2$] (a3) at ([rotate = 240] a1);
  \draw (0,-25pt) node {$w=\rho$};
  \draw (a2) -- (a3);
  \foreach \p in {a2,a3} \filldraw [black] (\p) circle (1.5pt);
 \end{scope}
 \begin{scope}[xshift=-3.5cm,yshift=2.5cm]
  \coordinate (a1) at (0,10pt);
  \coordinate [label=180:$1$] (a2) at ([rotate = 120] a1);
  \coordinate [label=0:$2$] (a3) at ([rotate = 240] a1);
  \draw (0,-25pt) node {$w=1$};
  \foreach \p in {a2,a3} \filldraw [black] (\p) circle (1.5pt);
 \end{scope}
 \begin{scope}
  \coordinate [label=90:$1$] (a1) at (0,10pt);
  \coordinate [label=180:$2$] (a2) at ([rotate = 120] a1);
  \coordinate [label=0:$3$] (a3) at ([rotate = 240] a1);
  \draw (0,-25pt) node {$w=\rho^3$};
  \draw (a1) -- (a2) -- (a3) -- (a1);
  \foreach \p in {a1,a2,a3} \filldraw [black] (\p) circle (1.5pt);
 \end{scope}
 \begin{scope}[xshift=2.5cm]
  \coordinate [label=90:$1$] (a1) at (0,10pt);
  \coordinate [label=180:$2$] (a2) at ([rotate = 120] a1);
  \coordinate [label=0:$3$] (a3) at ([rotate = 240] a1);
  \draw (0,-25pt) node {$w=\rho^2$};
  \draw (a2) -- (a3) -- (a1);
  \foreach \p in {a1,a2,a3} \filldraw [black] (\p) circle (1.5pt);
 \end{scope}
 \begin{scope}[xshift=5cm]
  \coordinate [label=90:$1$] (a1) at (0,10pt);
  \coordinate [label=180:$2$] (a2) at ([rotate = 120] a1);
  \coordinate [label=0:$3$] (a3) at ([rotate = 240] a1);
  \draw (0,-25pt) node {$w=\rho^2$};
  \draw (a2) -- (a1) -- (a3);
  \foreach \p in {a1,a2,a3} \filldraw [black] (\p) circle (1.5pt);
 \end{scope}
 \begin{scope}[xshift=7.5cm]
  \coordinate [label=90:$1$] (a1) at (0,10pt);
  \coordinate [label=180:$2$] (a2) at ([rotate = 120] a1);
  \coordinate [label=0:$3$] (a3) at ([rotate = 240] a1);
  \draw (0,-25pt) node {$w=\rho^2$};
  \draw (a1) -- (a2) -- (a3);
  \foreach \p in {a1,a2,a3} \filldraw [black] (\p) circle (1.5pt);
 \end{scope}
 \begin{scope}[xshift=0cm, yshift=2.5cm]
  \coordinate [label=90:$1$] (a1) at (0,10pt);
  \coordinate [label=180:$2$] (a2) at ([rotate = 120] a1);
  \coordinate [label=0:$3$] (a3) at ([rotate = 240] a1);
  \draw (0,-25pt) node {$w=1$};
  \foreach \p in {a1,a2,a3} \filldraw [black] (\p) circle (1.5pt);
 \end{scope}
 \begin{scope}[xshift=2.5cm, yshift=2.5cm]
  \coordinate [label=90:$1$] (a1) at (0,10pt);
  \coordinate [label=180:$2$] (a2) at ([rotate = 120] a1);
  \coordinate [label=0:$3$] (a3) at ([rotate = 240] a1);
  \draw (0,-25pt) node {$w=\rho$};
  \draw (a1) -- (a2);
  \foreach \p in {a1,a2,a3} \filldraw [black] (\p) circle (1.5pt);
 \end{scope}
 \begin{scope}[xshift=5cm, yshift=2.5cm]
  \coordinate [label=90:$1$] (a1) at (0,10pt);
  \coordinate [label=180:$2$] (a2) at ([rotate = 120] a1);
  \coordinate [label=0:$3$] (a3) at ([rotate = 240] a1);
  \draw (0,-25pt) node {$w=\rho$};
  \draw (a3) -- (a2);
  \foreach \p in {a1,a2,a3} \filldraw [black] (\p) circle (1.5pt);
 \end{scope}
 \begin{scope}[xshift=7.5cm, yshift=2.5cm]
  \coordinate [label=90:$1$] (a1) at (0,10pt);
  \coordinate [label=180:$2$] (a2) at ([rotate = 120] a1);
  \coordinate [label=0:$3$] (a3) at ([rotate = 240] a1);
  \draw (0,-25pt) node {$w=\rho$};
  \draw (a1) -- (a3);
  \foreach \p in {a1,a2,a3} \filldraw [black] (\p) circle (1.5pt);
 \end{scope}
\end{tikzpicture}
\end{center}
\caption{Weights of the labeled graphs whose size does not exceed $3$.}\label{figure: small graph weights}
\end{figure}
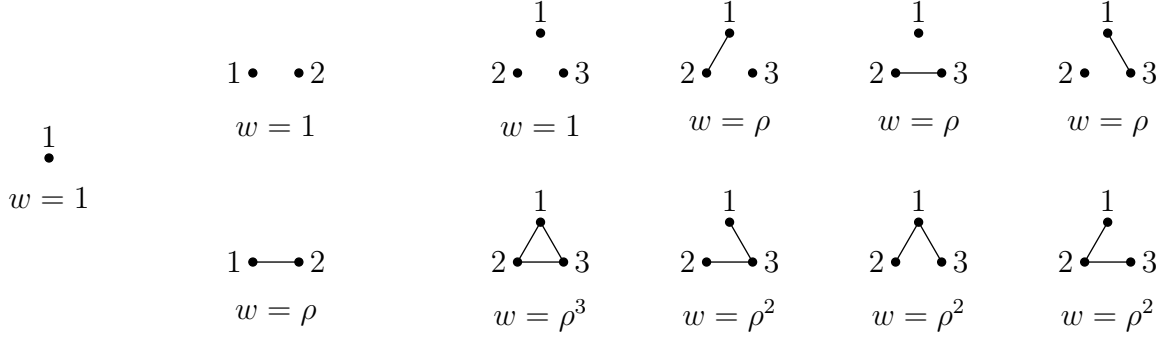

The following two facts will be useful for our investigation.
First, the species $\mcG$ satisfies the relation
 \begin{equation}\label{formula: graph decomposition}
  \mcG = \mcE \circ\mcCG,
 \end{equation}
where $\mcCG\subset\mcG$ is the subspecies of \emph{connected graphs}.
Second, the total weight of the species of graphs on $[n]$ is $(\rho+1)^{\binom{n}{2}} = q^{-\binom{n}{2}}$.

We will also need the following lemma.

\begin{lemma}\label{lemma: Erdos-Renyi is gargantuan}
  The weighted species $\mcG$ of graphs is gargantuan.
\end{lemma}
\begin{proof}
 According to Definition~\ref{def: gargantuan species}, our goal is to show that the sequence
 \[
  a_n = \dfrac{\big|\mcG[n]\big|}{n!} = \dfrac{(\rho+1)^{\binom{n}{2}}}{n!}
 \]
 is gargantuan.
 The proof idea is to apply Lemma~\ref{lemma: sufficient conditions for gargantuan sequence} by checking its two conditions.
 The first of them holds, since
 \[
  \dfrac{na_{n-1}}{a_n} =
  \dfrac{n}{(\rho+1)^n} \to 0,
 \]
 as $n\to\infty$.
 In order to verify the second one, let us show that the sequence
 \[
  x_k = a_ka_{n-k}
 \]
 decreases for $k < n/2$.
 For large $n$, we have
 \[
  \dfrac{x_{k+1}}{x_k} \le 1
  \quad\Leftrightarrow\quad
  \dfrac{(\rho+1)^{k+1}}{(k+1)} \le \dfrac{(\rho+1)^{n-k}}{(n-k)}
  \quad\Leftrightarrow\quad
  {k+1} \le {n-k},
 \]
 since the function
 \[
  f(x) = \dfrac{(\rho+1)^x}{x}
 \]
 is increasing for large $x$.
 Therefore, Lemma~\ref{lemma: sufficient conditions for gargantuan sequence} is applicable and $\mcG$ is gargantuan.
\end{proof}

\subsubsection{Asymptotic probability of connected graphs}

\begin{theorem}\label{theorem: Erdos-Renyi asymptotics}
  Let $m$ be a fixed positive integer, and $p\in(0,1)$.
 The asymptotic probability that a random graph~$g$ in the Erd\H{o}s--Rényi model~$G(n,p)$ has $m$ connected components satisfies 
 \begin{equation}\label{formula: Erdos-Renyi SET_m-asymptotics}
  \bbP(g\mbox{ has }m\mbox{ connected components})
   \approx
  \sum\limits_{k\ge0} P_{k,m}(\rho) \cdot \binom{n}{k} \cdot \dfrac{q^{nk}}{q^{k(k+1)/2}},
 \end{equation}
 where
 \[
  P_{k,m}(\rho) = \sum\limits_{g'\in\mcG_k}(-1)^{\pi_0(g')-(m-1)}\binom{\pi_0(g')}{m-1}w(g')
 \]
 and $\pi_0(g')$ is the number of connected components of the graph $g'$.
\end{theorem}
\begin{proof}
 The main idea is to apply Theorem~\ref{theorem: species E asymptotics} to the case where
 \(
  \mcG = \mcE\circ\mcCG.
 \)
 This is possible, since the species $\mcG$ is gargantuan due to Lemma~\ref{lemma: Erdos-Renyi is gargantuan},
 and Theorem~\ref{theorem: species E asymptotics} gives us
 \[
  \bbP(g\in\mcE_m\circ\mcCG)
   \approx
  \sum\limits_{k\ge0}P_{k,m}(\rho)\cdot\binom{n}{k}\cdot\dfrac{q^{nk}}{q^{k(k+1)/2}},
 \]
 where $P_{k,m}(\rho)$ is the total weight of the species $\mcCG^{\{m-1\}}(\mcE^{-1}\circ\mcCG)$ on $[k]$.
 Now to finish the proof, we take into account that $\mcE_m\circ\mcCG=\mcCG^{\{m\}}$ is the species of graphs with exactly $m$ connected components and that
 \[
  \mcE^{-1}\circ\mcCG
   \equiv 
  \One - \mcCG + \mcCG^{\{2\}} - \mcCG^{\{3\}} + \ldots
 \]
 (see Lemma~\ref{lemma: E^(-1) = sum((-1)^(k)E_k)}).
\end{proof}

\begin{corollary}
 The dominant term in asymptotics~\eqref{formula: Erdos-Renyi SET_m-asymptotics} is
  \[
  \bbP(g\mbox{ has }m\mbox{ connected components}) = \binom{n}{m-1}\cdot \dfrac{q^{n(m-1)}}{q^{m(m-1)/2}} + O(n^mq^{mn}).
 \]
\end{corollary}
\begin{proof}
 This follows from Corollary~\ref{cor: leading term of species E asymptotics} and the fact that $q^{-\binom{1}{2}}=1$.
\end{proof}

For small values of $m$ and $k$, polynomials $P_{k,m}(\rho)$ are listed in Table~\ref{table: polynomials P_k^m}.
We can see that the sum of the coefficients in the $k$th column is zero for every $k>0$,
while the sum of the zeroth column is $1$.
This observation can be explained by the fact that the sum of probabilities of the form~\eqref{formula: Erdos-Renyi SET_m-asymptotics} taken over all positive integers $m$ is equal to $1$.
That is to say, for any graph there is a unique $m\in\bbZ_{>0}$ such that $m$ is the number of connected components of this graph.
 
\begin{table}[ht!]
 \[
  \small
  \begin{array}{c|rrrrr}
   k & 0 & 1 & 2 & 3 & 4 \\
   \hline
  P_{k,1}(\rho) & 1 & -1 & - \rho + 1 & - \rho^3 - 3\rho^2 + 3\rho - 1 & - \rho^6 - 6\rho^5 - 15\rho^4 - 12\rho^3 + 15\rho^2 - 6\rho + 1 \\
  P_{k,2}(\rho) & 0 & 1 & \rho - 2 & \rho^3 + 3\rho^2 - 6\rho + 3 & \rho^6 + 6\rho^5 + 15\rho^4 + 8\rho^3 - 30\rho^2 + 18\rho - 4 \\
  P_{k,3}(\rho) & 0 & 0 & 1 & 3\rho - 3 & 4\rho^3 + 15\rho^2 - 18\rho + 6 \\
  P_{k,4}(\rho) & 0 & 0 & 0 & 1 & 6\rho - 4 \\
  P_{k,5}(\rho) & 0 & 0 & 0 & 0 & 1 \\
  \end{array}
 \]
 \caption{Polynomials $P_{k,m}(\rho)$ for $k\le4$ and $m\le5$.}
 \label{table: polynomials P_k^m}
\end{table}

\begin{example}\label{ex: polynomials P_k as graphs}
 Observe the outcome of Theorem~\ref{theorem: Erdos-Renyi asymptotics} for connected graphs, that is, for $m=1$.
 In this case, the statement reads:
 the asymptotic probability that a random graph~$g$ in the Erd\H{o}s--Rényi model~$G(n,p)$ is connected satisfies 
 \begin{equation}\label{formula: Erdos-Renyi SET-asymptotics}
  \bbP(g\mbox{ is connected}) \approx
  1 - \sum\limits_{k\ge1} P_k(\rho)\cdot \binom{n}{k}\cdot \dfrac{q^{nk}}{q^{k(k+1)/2}},
 \end{equation}
 where
 \begin{equation}\label{formula: Erdos-Renyi SET-asymptotics coefficients}
  P_k(\rho) = -P_{k,1}(\rho) = \sum\limits_{g'\in\mcG_k}(-1)^{\pi_0(g')-1}w(g').
 \end{equation}
 In other words, the coefficient $P_k(\rho)$ is the sum of the weights of all (virtual) graphs of size $k$,
 where a weight is taken positive if and only if the number of connected components of the corresponding graph is odd (otherwise, the weight is negative).
 It can be seen from Figure~\ref{figure: small graph weights} that
 \[
  P_1(\rho) = 1,
  \qquad
  P_2(\rho) = \rho - 1,
  \qquad
  P_3(\rho) = \rho^3 + 3\rho^2 - 3\rho + 1,
 \]
 and hence,
 \begin{multline*}
  \bbP(g\mbox{ is connected}) =\\
  1 - \binom{n}{1}q^{n-1} - (\rho-1)\binom{n}{2}q^{2n-3} - (\rho^3 + 3\rho^2 - 3\rho + 1)\binom{n}{3}q^{3n-6} + O(n^4q^{4n}).
 \end{multline*}
\end{example}

\subsubsection{Weighted species of tournaments with ties}

The asymptotic probability coefficients $P_k(\rho)$ described by relation~\eqref{formula: Erdos-Renyi SET-asymptotics} have a clear combinatorial meaning in the case where $\rho=1$, that is, in the case where graphs of the same size are taken uniformly, with the same probability.
Namely~\cite{MonteilNurligareev2021},
 \[
  P_k(1) = \it_k
 \]
is the number of irreducible tournaments\footnote{Reminder: a \emph{tournament} is a directed graph in which each pair of (distinct) vertices is joined by exactly one directed edge.
A tournament is \emph{irreducible} if there is no partition of its vertices into two nonempty parts $A$ and $B$ such that every pair of vertices $(a,b) \in A \times B$ is joined by an edge directed from~$a$ to~$b$.
Equivalently, a tournament is \emph{irreducible} if it is strongly connected.}
of size $k$.
The goal of this section is to generalize the concept of irreducible tournament, so that it fits the interpretation of asymptotic coefficients within the Erd\H{o}s-R\'enyi model for any positive $\rho$.

A clue to this issue consists in passing to a discrete generalization first.
For any $d\geqslant 1$, we consider labeled $d$-multigraphs and $d$-multitournaments:
\begin{itemize}
 \item 
  in a \emph{$d$-multigraph}, each pair of (distinct) vertices is joined by at most $d$ indistinguishable edges,
 \item 
  in a \emph{$d$-multitournament}, each pair of vertices $i\ne j$ is joined by $d$ directed edges
  (for some $\ell\in\{0,\ldots,d\}$ that varies from pair to pair,
  $\ell$ of $d$ edges are directed from $i$ to $j$, the other $d-\ell$ edges are directed from $j$ to $i$,
  and all co-directed edges are indistinguishable).
\end{itemize}
As we showed in~\cite{MonteilNurligareevSET},
the asymptotic probability that a random labeled $d$-multigraph $g$ is connected satisfies
 \begin{equation}\label{formula: SET-asymptotics for multigraphs}
  \bbP\big(g\mbox{ is connected}\big)
   \approx
  1
   - 
  \sum\limits_{k\ge1} \it_k(d) \cdot \binom{n}{k} \cdot \frac{(d+1)^{k(k+1)/2}}{(d+1)^{kn}},
 \end{equation}
where $\it_k(d)$ counts irreducible $d$-multitournaments of size $k$, that is, $d$-multitournaments that are strongly connected as directed graphs.
In particular, a $1$-multigraph is a (simple) graph, a $1$-multitournament is a tournament, and $\it_k(1)=\it_k$.

On the other hand, in terms of connectivity, a $d$-multigraph of size $n$ can be viewed as a graph within the Erd\H{o}s-R\'enyi model $G(n,p)$ with
 \[
  p = \dfrac{d}{d+1}.
 \]
This observation makes it natural to interpret $d$-multitournaments as directed graphs in which each pair of vertices $i\ne j$ is independently joined by an edge $\vec{ij}$, an edge $\vec{ji}$, both edges $\vec{ij}$ and $\vec{ji}$, respectively, with probability $1/(d+1)$, $1/(d+1)$, and $(d-1)/(d+1)$.

Now we are ready to define tournaments with ties within the Erd\H{o}s-R\'enyi model.

\begin{definition}\label{def: tournaments with ties}
 Let $p\in[1/2,1]$ be a fixed real number, $q=1-p$, and $n\in\bbZ_{\geqslant0}$.
 A~\emph{tournament with ties} within the \emph{Erd\H{o}s--Rényi model} $T(n,p)$ is a directed graph whose set of vertices is $[n]$ and each pair of vertices $i\ne j$ is joined independently by
\begin{itemize}
 \item 
  an edge directed from $i$ to $j$ with probability $q$,
 \item 
  an edge directed from $j$ to $i$ with probability $q$,
 \item 
  two edges, directed both from $i$ to $j$ and from $j$ to $i$, with probability $p-q$.
\end{itemize}
\end{definition}

\begin{remark}\label{remark: generalized tournaments}
 In the literature, the tournaments are usually generalized in the other way.
 Specifically, a \emph{generalized tournament} is understood as a directed graph in which an edge directed from $i$ to $j$ is taken with probability $p_{ij}$, where $p_{ij}+p_{ji}=1$ for $i\ne j$ and $p_{ii}=0$; see~\cite{Moon1968}.
 In particular, the \emph{Erd\H{o}s--Rényi model} is typically defined as the case $p_{ij}=p$ for $i<j$ and $p_{ij}=q$ for $i>j$, where positive values $p$ and $q$ satisfy $p+q=1$.
 This model concerns counting by descents.
 However, the relation~\eqref{formula: tournament decomposition} is not valid in this case; see~\cite{ArcherGesselGravesLiang2020}.
\end{remark}

\

Clearly, we cannot extend Definition~\ref{def: tournaments with ties} to $p<1/2$, since the probability value $p-q$ becomes negative.
To avoid this obstacle, let us switch to virtual species of tournaments with ties.
Thus, introduce the \emph{weighted species $\mcT$ of tournaments with ties} to be the species of directed graphs in which every pair of vertices is joined by one directed edge or by two oppositely directed edges.
For a tournament $t\in\mcT$, we assign a weight 
 \[
  w(t) = (\rho-1)^k,
 \]
where $\rho = p/q$ and $k$ is the number of pairs of vertices that are joined by two edges.
Note that for $\rho<1$, the weight $w(t)$ may be negative.
In this case, the corresponding tournaments are virtual.
Some examples of small-size tournaments with ties are presented in Figure~\ref{figure: small tournaments weights} (among them, the irreducible tournaments are depicted at the bottom of the illustration, while the reducible ones can be found at the top).

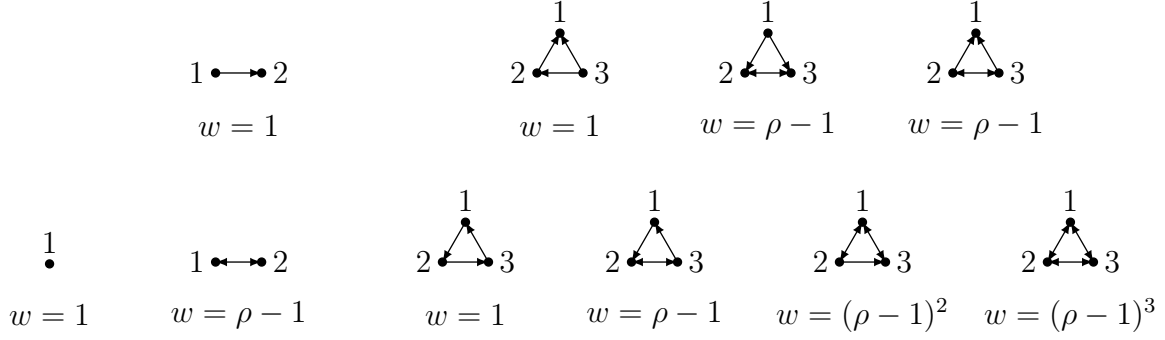
\begin{figure}[!ht]
\begin{center}
\begin{tikzpicture}[>=latex, line width=.5pt]
 \begin{scope}[xshift=-5.5cm, yshift=-0.2cm]
  \coordinate [label=90:$1$] (a1) at (0,0);
  \draw (0,-19pt) node {$w=1$};
  \filldraw [black] (a1) circle (1.5pt);
 \end{scope}
 \begin{scope}[xshift=-3cm, yshift=2.5cm]
  \coordinate (a1) at (0,10pt);
  \coordinate [label=180:$1$] (a2) at ([rotate = 120] a1);
  \coordinate [label=0:$2$] (a3) at ([rotate = 240] a1);
  \draw[->] (a2) -- (a3);
  \draw (0,-25pt) node {$w=1$};
  \foreach \p in {a2,a3} \filldraw [black] (\p) circle (1.5pt);
 \end{scope}
 \begin{scope}[xshift=-3cm]
  \coordinate (a1) at (0,10pt);
  \coordinate [label=180:$1$] (a2) at ([rotate = 120] a1);
  \coordinate [label=0:$2$] (a3) at ([rotate = 240] a1);
  \draw (0,-25pt) node {$w=\rho-1$};
  \draw[<->] (a2) -- (a3);
  \foreach \p in {a2,a3} \filldraw [black] (\p) circle (1.5pt);
 \end{scope}
 \begin{scope}[xshift=1.25cm, yshift=2.5cm]
  \coordinate [label=90:$1$] (a1) at (0,10pt);
  \coordinate [label=180:$2$] (a2) at ([rotate = 120] a1);
  \coordinate [label=0:$3$] (a3) at ([rotate = 240] a1);
  \draw (0,-25pt) node {$w=1$};
  \draw[<-] (a1) -- (a2);
  \draw[<-] (a1) -- (a3);
  \draw[<-] (a2) -- (a3);
  \foreach \p in {a1,a2,a3} \filldraw [black] (\p) circle (1.5pt);
 \end{scope}
 \begin{scope}[xshift=4cm, yshift=2.5cm]
  \coordinate [label=90:$1$] (a1) at (0,10pt);
  \coordinate [label=180:$2$] (a2) at ([rotate = 120] a1);
  \coordinate [label=0:$3$] (a3) at ([rotate = 240] a1);
  \draw (0,-25pt) node {$w=\rho-1$};
  \draw[->] (a1) -- (a2);
  \draw[->] (a1) -- (a3);
  \draw[<->] (a2) -- (a3);
  \foreach \p in {a1,a2,a3} \filldraw [black] (\p) circle (1.5pt);
 \end{scope}
 \begin{scope}[xshift=6.75cm, yshift=2.5cm]
  \coordinate [label=90:$1$] (a1) at (0,10pt);
  \coordinate [label=180:$2$] (a2) at ([rotate = 120] a1);
  \coordinate [label=0:$3$] (a3) at ([rotate = 240] a1);
  \draw (0,-25pt) node {$w=\rho-1$};
  \draw[<-] (a1) -- (a2);
  \draw[<-] (a1) -- (a3);
  \draw[<->] (a2) -- (a3);
  \foreach \p in {a1,a2,a3} \filldraw [black] (\p) circle (1.5pt);
 \end{scope}
 \begin{scope}[xshift=0cm, yshift=0cm]
  \coordinate [label=90:$1$] (a1) at (0,10pt);
  \coordinate [label=180:$2$] (a2) at ([rotate = 120] a1);
  \coordinate [label=0:$3$] (a3) at ([rotate = 240] a1);
  \draw (0,-25pt) node {$w=1$};
  \draw[->] (a1) -- (a2);
  \draw[<-] (a1) -- (a3);
  \draw[->] (a2) -- (a3);
  \foreach \p in {a1,a2,a3} \filldraw [black] (\p) circle (1.5pt);
 \end{scope}
 \begin{scope}[xshift=2.5cm]
  \coordinate [label=90:$1$] (a1) at (0,10pt);
  \coordinate [label=180:$2$] (a2) at ([rotate = 120] a1);
  \coordinate [label=0:$3$] (a3) at ([rotate = 240] a1);
  \draw (0,-25pt) node {$w=\rho-1$};
  \draw[->] (a1) -- (a2);
  \draw[<-] (a1) -- (a3);
  \draw[<->] (a2) -- (a3);
  \foreach \p in {a1,a2,a3} \filldraw [black] (\p) circle (1.5pt);
 \end{scope}
 \begin{scope}[xshift=5.25cm]
  \coordinate [label=90:$1$] (a1) at (0,10pt);
  \coordinate [label=180:$2$] (a2) at ([rotate = 120] a1);
  \coordinate [label=0:$3$] (a3) at ([rotate = 240] a1);
  \draw (0,-25pt) node {$w=(\rho-1)^2$};
  \draw[<->] (a1) -- (a2);
  \draw[<->] (a1) -- (a3);
  \draw[->] (a2) -- (a3);
  \foreach \p in {a1,a2,a3} \filldraw [black] (\p) circle (1.5pt);
 \end{scope}
 \begin{scope}[xshift=8cm]
  \coordinate [label=90:$1$] (a1) at (0,10pt);
  \coordinate [label=180:$2$] (a2) at ([rotate = 120] a1);
  \coordinate [label=0:$3$] (a3) at ([rotate = 240] a1);
  \draw (0,-25pt) node {$w=(\rho-1)^3$};
  \draw[<->] (a1) -- (a2);
  \draw[<->] (a2) -- (a3);
  \draw[<->] (a1) -- (a3);
  \foreach \p in {a1,a2,a3} \filldraw [black] (\p) circle (1.5pt);
 \end{scope}
\end{tikzpicture}
\end{center}
\caption{Weights of some labeled tournaments with ties.
}\label{figure: small tournaments weights}
\end{figure}

In view of the asymptotic probability of connected graphs, the following two observations are crucial.
First, the species $\mcT$ can be represented as a composition
 \begin{equation}\label{formula: tournament decomposition}
  \mcT = \mcL \circ \mcIT,
 \end{equation}
where $\mcIT \subset \mcT$ is the subspecies of irreducible (that is, strongly connected) tournaments with ties.
The proof of this relation is essentially the same as for standard tournaments (see Lemma~1 in~\cite{MonteilNurligareev2021}): any tournament with ties can be viewed as a sequence of its strongly connected components.
Second, the total weight of the species of tournaments with ties on the set $[n]$ is the same as for the species $\mcG$ and equals $(\rho+1)^{\binom{n}{2}}$.
This means that the species $\mcG$ of graphs and $\mcT$ of tournaments with ties are equipotent.

Taking into account the observations above, we can propose the following interpretation of relation~\eqref{formula: Erdos-Renyi SET-asymptotics}.

\begin{theorem}\label{theorem: Erdos-Renyi asymptotics in terms of tournaments}
 The asymptotic probability that a random graph~$g$ in the Erd\H{o}s--Rényi model~$G(n,p)$ is connected satisfies 
 \[
  \bbP(g\mbox{ is connected})
   \approx
  1 - \sum\limits_{k\ge1} P_k(\rho) \cdot \binom{n}{k} \cdot \dfrac{q^{nk}}{q^{k(k+1)/2}},
 \]
 where $P_k(\rho)=\big|\mcIT[k]\big|$ is the total weight of species $\mcIT$ on $[k]$.
 In particular, if $p\ge1/2$ (that is, $\rho\ge1$), then $P_k(\rho)\ge0$ for any positive integer $k$.
\end{theorem}
\begin{proof}
 According to Theorem~\ref{theorem: Erdos-Renyi asymptotics}, polynomials $P_k(\rho)$ represent the total weight of the species $\mcE^{-1}\circ\mcCG$ on $[k]$.
 Due to Lemma~\ref{lemma: magical species identity}, this species is equipotent to the species $\big(\One-\mcL_+^{(-1)}\big)\circ\mcG_+$.
 The last one is equipotent to $\One-\mcIT$, since, as we have seen above, the species $\mcG_+$ and $\mcT_+$ are equipotent, and $\mcT_+ = \mcL_+ \circ \mcIT$.
 Therefore, $P_k(\rho)$ can be interpreted as the total weight of the species of irreducible tournaments with ties on $[k]$.
\end{proof}

\begin{example}\label{ex: polynomials P_k as tournaments}
  As we discussed in Example~\ref{ex: polynomials P_k as graphs}, the first three polynomials $P_k(\rho)$ are
 \[
  P_1(\rho) = 1,
  \qquad
  P_2(\rho) = \rho - 1,
  \qquad
  P_3(\rho) = \rho^3 + 3\rho^2 - 3\rho + 1.
 \]
 The first two correspond to unique irreducible tournaments with ties of size 1 and 2, respectively (see Figure~\ref{figure: small tournaments weights}).
 The third polynomial is equal to the sum of all irreducible tournaments with ties of size 3:
 \[
  P_3(\rho)
   =
  (\rho-1)^3 + 6(\rho-1)^2 + 6(\rho-1) + 2.
 \]
\end{example}

\subsubsection{Asymptotic probability of irreducible tournaments with ties}

Similarly to the main question stated for the Erd\H{o}s-R\'enyi model $G(n,p)$ of random graphs,
we could ask about the asymptotic probability that a random tournament with ties is irreducible within the Erd\H{o}s-R\'enyi model $T(n,p)$ introduced in Definition~\ref{def: tournaments with ties} for $p\ge1/2$.
In this section, we answer a more general question.
Namely, we consider virtual weighted species $\mcT$ of tournaments with ties and establish the complete asymptotic expansion of the ratio $\it_n^{(m)}/\it_n$,
where $m$ is an arbitrary positive integer,
while $\it_n=\it_n(\rho)$ and $\it_n^{(m)}=\it_n^{(m)}(\rho)$ are the total weights on the set $[n]$, respectively,
of virtual species of all tournaments with ties and its subspecies of tournaments consisting of $m$ strongly connected components.
This result is given by the following theorem.

\begin{theorem}\label{theorem: tournaments with ties asymptotics}
 Let $m$ be a fixed positive integer.
 The virtual weighted species $\mcT$ satisfies
 \begin{equation}\label{formula: tournaments with ties asymptotics}
  \dfrac{\it_n^{(m)}}{\it_n}
   \approx
  m \sum\limits_{k\ge0} \Big(\it_k^{(m-1)} - 2\it_k^{(m)} + \it_k^{(m+1)}\Big) \cdot \binom{n}{k} \cdot \dfrac{q^{nk}}{q^{k(k+1)/2}},
 \end{equation}
 where, by convention, $\it_0^{(0)}=1$ and $\it_k^{(0)}=0$ for $k>0$.
\end{theorem}
\begin{proof}
 As follows from Lemma~\ref{lemma: Erdos-Renyi is gargantuan}, the species $\mcT$ is gargantuan.
 Hence, we can apply Theorem~\ref{theorem: species L asymptotics} to the decomposition $\mcT = \mcL\circ\mcIT$ (see Remark~\ref{rem: asymptotic theorem for virtual species}).
 This gives us
 \[
   \dfrac{\it_n^{(m)}}{\it_n}
   \approx
  \sum\limits_{k\ge0} \d_{k,m}(\mcT) \cdot \binom{n}{k} \cdot \dfrac{q^{nk}}{q^{k(k+1)/2}},
 \]
 where $\d_{k,m}(\mcT)$ is the total weight of the species of structures $m\,\mcIT^{m-1}(1-\mcIT)^2$, that is,
 \[
  \d_{k,m}(\mcT) = m \Big(\it_k^{(m-1)} - 2\it_k^{(m)} + \it_k^{(m+1)}\Big).
\]
\end{proof}

\begin{corollary}\label{cor: Erdos-Renyi T(n,p) asymptotics}
 Let $m$ be a fixed positive integer, and $p\in[1/2,1)$.
 The asymptotic probability that a random tournament with ties~$t$ within the Erd\H{o}s--Rényi model~$T(n,p)$ has $m$ strongly connected components satisfies 
 \begin{equation}\label{formula: Erdos-Renyi SEQ_m-asymptotics}
  \bbP(t\mbox{ has }m\mbox{ strongly connected components})
   \approx
  \sum\limits_{k\ge0} Q_{k,m}(\rho) \cdot \binom{n}{k} \cdot \dfrac{q^{nk}}{q^{k(k+1)/2}},
 \end{equation}
 where
 \[
  Q_{k,m}(\rho) = m\Big(\it_k^{(m-1)} - 2\it_k^{(m)} + \it_k^{(m+1)}\Big).
 \]
 In particular,
 \begin{equation}\label{formula: Erdos-Renyi SEQ-asymptotics}
  \bbP(t\mbox{ is strongly connected})
   \approx
  1 - \sum\limits_{k\ge0} \Big(2\it_k - \it_k^{(2)}\Big) \cdot \binom{n}{k} \cdot \dfrac{q^{nk}}{q^{k(k+1)/2}}.
 \end{equation}
\end{corollary}

\begin{corollary}
 The dominant term in asymptotics~\eqref{formula: Erdos-Renyi SEQ_m-asymptotics} is
  \[
  \bbP(t\mbox{ has }m\mbox{ strongly connected components}) = m\cdot(n)_{m-1}\cdot \dfrac{q^{n(m-1)}}{q^{m(m-1)/2}} + O(n^mq^{mn}),
 \]
 where $(n)_m = n(n-1)\ldots(n-m+1)$ are the falling factorials.
\end{corollary}
\begin{proof}
 This follows from Corollary~\ref{cor: leading term of species L asymptotics} and the fact that $q^{-\binom{1}{2}}=1$.
\end{proof}

It is easy to see that the results obtained in the sections are in accordance with~\cite{MonteilNurligareev2021} and~\cite{MonteilNurligareevSEQ} that presented the particular case $p=1/2$.
In Table~\ref{table: polynomials Q_k^m}, the reader can find the polynomials $Q_{k,m}(\rho)$ for small values of $m$ and $k$.
Similarly to polynomials~$P_{k,m}(\rho)$ (see Table~\ref{table: polynomials P_k^m}),
the sum of the coefficients in the $k$th column is zero for every $k>0$,
and the sum of the zeroth column is $1$.
This is due to the fact that the sum of probabilities of the form~\eqref{formula: Erdos-Renyi SEQ_m-asymptotics} taken over all positive integers $m$ is equal to $1$.
In other words, for any tournament with ties, there is a unique $m\in\bbZ_{>0}$ such that $m$ is the number of its strongly connected components.
 
\begin{table}[ht!]
 \[
  \small
  \begin{array}{c|rrrrr}
   k & 0 & 1 & 2 & 3 & 4 \\
   \hline
  Q_{k,1}(\rho) & 1 & -2 & - 2\rho + 4 & - 2\rho^3 - 6\rho^2 + 12\rho - 8 & - 2\rho^6 - 12\rho^5 - 30\rho^4 - 16\rho^3 + 60\rho^2 - 48\rho + 16 \\
  Q_{k,2}(\rho) & 0 & 2 & 2\rho - 10 & 2\rho^3 + 6\rho^2 - 30\rho + 38 & 2\rho^6 + 12\rho^5 + 30\rho^4 - 8\rho^3 - 150\rho^2 + 228\rho - 130 \\
  Q_{k,3}(\rho) & 0 & 0 & 6 & 18\rho - 54 & 24\rho^3 + 90\rho^2 - 324\rho + 330 \\
  Q_{k,4}(\rho) & 0 & 0 & 0 & 24 & 144\rho - 336 \\
  Q_{k,5}(\rho) & 0 & 0 & 0 & 0 & 120 \\
  \end{array}
 \]
 \caption{Polynomials $Q_{k,m}(\rho)$ for $k\le4$ and $m\le5$.}
 \label{table: polynomials Q_k^m}
\end{table}

\subsection{Discrete surfaces and manifolds}
\label{subsec: surface models}

This section is devoted to surface and manifold models whose asymptotic behavior could not be treated using the approach described in~\cite{MonteilNurligareevSET}, namely quadratic square-tiled surfaces, $P$-angulated surfaces and (possibly non-orientable) Graph Encoded Manifolds (GEMs).

For those models, we will establish the asymptotic probability that a random surface (or manifold) of size $n$ is connected or consists of a prescribed number of connected components, and interpret the obtained result in a combinatorial manner.
We will achieve these objectives using Theorem~\ref{theorem: species E asymptotics} and Proposition~\ref{prop: p-periodic species asymptotics} for the case where $\mcF=\mcE$.

\subsubsection{Quadratic square-tiled surfaces}
\label{sec: quadratic square-tiled surfaces}
Square-tiled surfaces play an important role in the study of moduli spaces of abelian and quadratic differentials on Riemann surfaces (see~\cite{DouadyHubbard1975,EskinOkounkov2001,EskinOkounkov2006,HerrlichSchmithuesen2008}).
In ~\cite{MonteilNurligareevSET}, we provided the whole structure of the asymptotic expansion of the probability that an abelian square-tiled surface is connected.
This class admits a double SET/SEQ decomposition, and the associated derivative class is the class of indecomposable permutations.
Quadratic square-tiled surfaces, however, seem not to admit such a natural derivative class.

A \emph{quadratic square-tiled surface} of size $n$ is an oriented surface obtained by gluing $n$ unit squares in such a~way that the horizontal sides are glued to the horizontal sides, while the vertical sides are glued to the vertical sides.
We assume that it is allowed to glue together the sides of the same square.
We also assume that the squares are distinguishable, that is, they are labeled from 1 to $n$ (or by elements of any other set of size $n$).
The species of quadratic square-tiled surface will be denoted by $\mcQSS$, while the total number of quadratic square-tiled surfaces of size $n$ will be denoted by $\qss_n$. 

From the definition of the model, it follows that $\qss_n = \big((2n-1)!!\big)^2$.
Another relation derived from this definition is the decomposition $\mcQSS = \mcE \circ \mcCQSS$, where $\mcCQSS\subset\mcQSS$ is the subspecies of connected quadratic square-tiled surfaces.
These two relations allow us to obtain a complete asymptotic expansion of the probability that a random quadratic square-tiled surface of size $n$ is connected (and more).

\begin{proposition}\label{prop: qss asymptotics}
 Let $m$ be a fixed positive integer.
 The asymptotic probability that a~random quadratic square-tiled surface $s\in\mcQSS$ of size $n$ has $m$ connected components satisfies
 \begin{equation}\label{formula: qss asymptotics}
  \bbP(s\mbox{ has }m\mbox{ connected components})
   \approx
  \sum\limits_{k\ge0}
   \d_{k,m}(\mcQSS)
    \cdot
   \binom{n}{k}
    \cdot
   \left(\dfrac{\big(2(n-k)-1\big)!!}{(2n-1)!!}\right)^2,
 \end{equation}
 where
 \begin{align*}
  \d_{k,m}(\mcQSS)
   & = 
  \#\big\{(s',s'')\in\mcQSS_i\cdot\mcQSS_{k-i} \mid 0\leqslant i \leqslant k,\,\, \pi_0(s')=m-1,\,\, \pi_0(s'') \mbox{ is even}\big\} \\
   & -
  \#\big\{(s',s'')\in\mcQSS_i\cdot\mcQSS_{k-i} \mid 0\leqslant i \leqslant k,\,\, \pi_0(s')=m-1,\,\, \pi_0(s'') \mbox{ is odd}\big\}
 \end{align*}
 and $\pi_0(s)$ is the number of connected components of the surface $s$.
\end{proposition}
\begin{proof}
 Since $\mcQSS = \mcE \circ \mcCQSS$, the proof idea is to apply Proposition~\ref{prop: d_(k,m) interpretation} to the species $\mcA=\mcQSS$.
 To this end, it is sufficient to check that the sequence $a_n=\qss_n/n!$ is gargantuan.
 This can be done using Lemma~\ref{lemma: sufficient conditions for gargantuan sequence}.
 As condition $na_{n-1} = O(a_n)$ is trivially satisfied, our goal is to verify that the sequence $x_n = a_ka_{n-k}$ is decreasing for $k<n/2$.
 Straightforward calculations show that
 \[
  x_k = \dfrac{\Big(\big(2k-1\big)!!\big(2(n-k)-1\big)!!\Big)^2}{k!(n-k)!},
 \]
 and the ratio
 \[
  \dfrac{x_{k+1}}{x_k} = \left(\dfrac{2k+1}{2(n-k)-1}\right)^2 \cdot \dfrac{n-k}{k+1}
 \]
 is less than $1$ if and only if
 \[
  \dfrac{\big(2(n-k)-1\big)^2}{n-k} \geqslant \dfrac{\big(2(k+1)-1\big)^2}{k+1}.
 \]
 Since the function $f(x) = (2x-1)^2/x$ is increasing for $x>1$, the latter inequality is equivalent to $(k+1)\leqslant(n-k)$.
 Hence, the sequence $(x_k)$ is decreasing for $k<n/2$, and the sequence $(a_n)$ is gargantuan.
 Thus, Theorem~\ref{theorem: species E asymptotics} is applicable in our case, and relation~\eqref{formula: species E asymptotics} converts into~\eqref{formula: qss asymptotics}, which completes our proof. 
\end{proof}

\begin{corollary}\label{cor: qss asymptotics, dominant}
 Let $m$ be a fixed positive integer.
 The dominant term of the asymptotic probability~\eqref{formula: qss asymptotics} that a random quadratic square-tiled surface $s\in\mcQSS$ of size $n$ has $m$~connected components is
 \begin{equation}\label{formula: qss asymptotics, dominant}
  \bbP(s\mbox{ has }m\mbox{ connected components})
   \sim
   \dfrac{1}{(4n)^{m-1}}.
 \end{equation}
 In particular,
 \[
  \bbP(s\mbox{ is not connected})
   \sim
   \dfrac{1}{4n}.
 \]
\end{corollary}
\begin{proof}
 Since $\qss_1=1$, it follows from Corollary~\ref{cor: leading term of species E asymptotics} that the dominant term of~\eqref{formula: qss asymptotics} is
 \[
  \binom{n}{m-1}
   \cdot
  \left(\dfrac{\big(2(n-m+1)-1\big)!!}{(2n-1)!!}\right)^2
   \sim
  \dfrac{1}{(4n)^{m-1}}.
 \]
\end{proof}

\subsubsection{\emph{P}-angulated surfaces}
\label{sec: p-angulations surfaces}

Combinatorial maps are obtained by gluing polygons of any perimeters along their edges, preserving the orientation. This class admits a double SET/SEQ decomposition, the associated derivative class is the class of indecomposable perfect matchings (or irreducible linear matchings), and we provided the whole structure of the asymptotic expansion of the probability that a combinatorial map is connected in~\cite{MonteilNurligareevSET}.
By fixing a perimeter $P$ for the polygons, we get triangulations (introduced in~\cite{BrooksMakover2004} to study ``typical'' Riemann surfaces of high genus and in~\cite{PippengerSchleich2005} for the needs of quantum gravity), quadrangulations, and more generally $P$-angulations (introduced in~\cite{Gamburd2006} in connection with the study of Belyi surfaces).
Such classes do not seem to admit a natural SET/SEQ decomposition, so the methods of~\cite{MonteilNurligareevSET} do not apply.

 Let us formally describe the model.
 Fix an integer $P\geqslant3$.
 A \emph{$P$-angulation} of size~$n$ is an oriented surface obtained by gluing $n$ unit $P$-gons along their sides.
 Here, we assume that $P$-gons are labeled from $1$ to $n$, and that their sides are distinguishable.
 At the same time, there is no restriction on the adhesion of the sides: any side can be glued to any of the $(Pn-1)$ other sides.
 The only restriction is combinatorial: a gluing is a~partition of polygon sides in pairs, and so the number $Pn$ must be even.
 We will denote the species of $P$-angulated surfaces and the total number of $P$-angulated surfaces of size $n$ by $\mcPS(P)$ and $\ps_n(P)$, respectively.
 From the definition, it follows immediately that $\ps_n(P) = (Pn-1)!!$ for even values of $P$, while for odd values of $P$ we have
 \[
  \ps_n(P) = \left\{
   \begin{array}{ll}
    (2Pk-1)!! & \mbox{if } n=2k \\ 
    0 & \mbox{if } n=2k+1. \\
   \end{array}
  \right.
 \]

 The quadratic and abelian square-tiled surfaces discussed in the previous section form a subset of quadrangulated surfaces.
 Similarly to these particular models, each $P$-angulated surface is a union of its connected components, which gives us the decomposition $\mcPS(P) = \mcE \circ \mcCPS(P)$, where by $\mcCPS(P)$ we denote the species of connected $P$-angulated surfaces.
 In this section, we establish the asymptotic probability that a~random $P$-angulated surface consists of a given number of connected components.
 As we shall see, the form of the formulas slightly depends on the parity of $P$.

\begin{proposition}\label{prop: ps(p) asymptotics}
 Let $m\geqslant1$ and $P\geqslant3$ be two fixed integers.
 Depending on the parity of the parameter $P$, the asymptotic probability that a random surface $s\in\mcPS=\mcPS(P)$ has $m$ connected components satisfies
 \begin{multline}\label{formula: ps(p) asymptotics}
  \bbP(s\mbox{ has }m\mbox{ connected components})
   \approx \\
  \left\{\begin{array}{ll}
   \sum\limits_{k\ge0} \d_{k,m}(\mcPS) \cdot \displaystyle\binom{n}{k} \cdot \dfrac{\big(P(n-k)-1\big)!!}{(Pn-1)!!} & \mbox{if }P\mbox{ is even}  \\
   \sum\limits_{k\ge0} \d_{2k,m}(\mcPS) \cdot \displaystyle\binom{2n}{2k} \cdot \dfrac{\big(2P(n-k)-1\big)!!}{(2Pn-1)!!} & \mbox{if }P\mbox{ is odd},  \\
  \end{array}\right.
 \end{multline}
 where
 \begin{align*}
  \d_{k,m}(\mcPS)
   & = 
  \#\big\{(s',s'')\in\mcPS_i\cdot\mcPS_{k-i} \mid 0\leqslant i\leqslant k,\,\, \pi_0(s')=m-1,\,\, \pi_0(s'') \mbox{ is even}\big\} \\
   & -
  \#\big\{(s',s'')\in\mcPS_i\cdot\mcPS_{k-i} \mid 0\leqslant i\leqslant k,\,\, \pi_0(s')=m-1,\,\, \pi_0(s'') \mbox{ is odd}\big\}
 \end{align*}
 and $\pi_0(s)$ is the number of connected components of the surface $s$.
\end{proposition}
\begin{proof}
 The relation $\mcPS(P) = \mcE \circ \mcCPS(P)$ suggests that we apply Proposition~\ref{prop: d_(k,m) interpretation}.
 To do so, it would suffice to verify that the species $\mcPS(P)$ is gargantuan.
 However, this fact is formally correct only for the case where $P$ is even, since for odd values of $P$ the sequence $\big(\ps_n(P)\big)$ is 2-periodic.
 Thus, depending on the parity of the parameter $P$, we reason slightly differently.
 
 In the case where $P$ is even, we show that the sequence $a_n=\ps_n(P)/n!$ is gargantuan using Lemma~\ref{lemma: sufficient conditions for gargantuan sequence}.
 Since condition $na_{n-1} = O(a_n)$ holds trivially, the only thing to check is that the sequence $x_n = a_ka_{n-k}$ is decreasing for $k<n/2$.
 The case in the hands reads
 \[
  x_k = \dfrac{\big(Pk-1\big)!!\big(P(n-k)-1\big)!!}{k!(n-k)!},
 \]
 and the ratio
 \[
  \dfrac{x_{k+1}}{x_k} = \dfrac{\big(P(k+1)-1\big)\ldots\big(Pk+1\big)}{\big(P(n-k)-1\big)\ldots\big(P(n-k-1)+1\big)} \cdot \dfrac{n-k}{k+1}
 \]
 is less than $1$ if and only if
 \[
  \dfrac{\big(P(n-k)-1\big)\big(P(n-k)-3\big)\ldots\big(P(n-k)-(P-1)\big)}{\big(P(k+1)-1\big)\big(P(k+1)-3\big)\ldots\big(P(k+1)-(P-1)\big)} \geqslant \dfrac{n-k}{k+1}.
 \]
 Since the function $\big(Px-1\big)\big(Px-3\big)\ldots\big(Px-(P-1)\big)/x$ is increasing for $x>1$, the latter inequality is equivalent to $(k+1)\leqslant(n-k)$.
 Hence, the sequence $(x_k)$ is decreasing for $k<n/2$, and the sequence $(a_n)$ is gargantuan.
 Proposition~\ref{prop: d_(k,m) interpretation} can therefore be applied to the species $\mcA = \mcPS(P)$, and this gives relation~\eqref{formula: ps(p) asymptotics}.
 
 In the case where $P$ is odd, we verify that the sequence $b_n=\ps_{2n}(P)/(2n)!$ is gargantuan.
 Again, we use Lemma~\ref{lemma: sufficient conditions for gargantuan sequence}, and the reasoning is similar: trivially, we have $nb_{n-1} = O(b_n)$ as $n\to\infty$, and the sequence 
 \[
  y_n = b_kb_{n-k} = 
  \dfrac{\big(2Pk-1\big)!!\big(2P(n-k)-1\big)!!}{(2k)!\big(2(n-k)\big)!}
 \]
 is decreasing for $k<n/2$.
 The latter arises from the fact that the inequality $y_{k+1}<y_k$ is equivalent to $(k+1)\leqslant(n-k)$ for large $n$, which, in turn, follows from the fact that the function $\big(2Px-1\big)\big(2Px-3\big)\ldots\big(2Px-(2P-1)\big)/\big(x(2x-1)\big)$ is increasing for $x>1$.
 
 Since the sequence $(b_n)$ is gargantuan, we can apply Proposition~\ref{prop: p-periodic species asymptotics} with $p=2$ to the species $\mcA = \mcPS(P)$ and $\mcF = \mcE$.
 As a consequence, we obtain relation~\eqref{formula: ps(p) asymptotics}, which completes the proof.
\end{proof}

\begin{corollary}\label{cor: ps(p) asymptotics, dominant}
 Let $m\geqslant1$ and $P\geqslant3$ be two fixed integers.
 Depending on the parity of the parameter $P$, the dominant term of the asymptotic probability~\eqref{formula: ps(p) asymptotics} that a random surface $s\in\mcPS=\mcPS(P)$ has $m$ connected components satisfies
 \begin{multline}\label{formula: ps(p) asymptotics, dominant}
  \bbP(s\mbox{ has }m\mbox{ connected components})
   \sim \\
  \left\{\begin{array}{rcl}
   \dfrac{1}{(m-1)!}\cdot\left(\dfrac{(P-1)!!}{P^{P/2}} \cdot \dfrac{1}{n^{P/2-1}}\right)^{m-1} & & \mbox{if }P\mbox{ is even}  \\
   \dfrac{1}{(m-1)!}\cdot\left(\dfrac{(2P-1)!!}{P^{P} \cdot 2^{P-1}} \cdot \dfrac{1}{n^{P-2}}\right)^{m-1} & & \mbox{if }P\mbox{ is odd}.  \\
  \end{array}\right.
 \end{multline}
 In particular,
 \[
  \bbP(s\mbox{ is not connected})
   \sim \\
  \left\{\begin{array}{rcl}
   \dfrac{(P-1)!!}{P^{P/2}} \cdot \dfrac{1}{n^{P/2-1}} & & \mbox{if }P\mbox{ is even}  \\
   \dfrac{(2P-1)!!}{P^{P} \cdot 2^{P-1}} \cdot \dfrac{1}{n^{P-2}} & & \mbox{if }P\mbox{ is odd}.  \\
  \end{array}\right.
 \]
\end{corollary}
\begin{proof}
 If $P$ is even, then the total number of a single $P$-gon gluings is $\ps_1(P)=(P-1)!!$.
 Hence, according to Corollary~\ref{cor: leading term of species E asymptotics}, the leading term of~\eqref{formula: ps(p) asymptotics} behaves as
 \[
  \binom{n}{m-1}
   \cdot 
  \dfrac{\big((P-1)!!\big)^{m-1} \cdot \big(P(n-m+1)-1\big)!!}{(Pn-1)!!}
   \sim
  \dfrac{1}{(m-1)!}
   \cdot 
  \left(\dfrac{(P-1)!!}{P^{P/2}} \cdot \dfrac{1}{n^{P/2-1}}\right)^{m-1}.
 \]
 If $P$ is odd, then we have $\ps_2(P) = (2P-1)!!$ gluings of two $P$-gons, and therefore, due to Corollary~\ref{cor: leading term of p-periodic species asymptotics} with $p=2$, the leading term of~\eqref{formula: ps(p) asymptotics} is equivalent to
 \[
  \dfrac{(2n)! \cdot \big((2P-1)!!\big)^{m-1} \cdot \big(2P(n-m+1)-1\big)!!}
  {(m-1)! \cdot 2^{m-1} \cdot \big(2(n-m+1)\big)!! \cdot (2Pn-1)!!}
   \sim
  \dfrac{1}{(m-1)!}\cdot\left(\dfrac{(2P-1)!!}{P^{P} \cdot 2^{P-1}} \cdot \dfrac{1}{n^{P-2}}\right)^{m-1}.
 \]
\end{proof}

\subsubsection{Graph Encoded Manifolds}
\label{sec: GEMs}

Graph Encoded Manifolds (GEMs) were introduced within ``crystallization theory'' in order to encode compact PL-manifolds~\cite{Pezzana1974,Pezzana1975,FerriGagliardiGrasselli1986}.
They reappeared in the encoding of colored tensor models seen as quantum gravity theories (see, for instance,~\cite{BonzomGurauRielloRivasseau2011, Gurau2011, GurauRyan2012, Witten2019}).
In ~\cite{MonteilNurligareevSET}, we provided the whole structure of the asymptotic expansion of the probability that an oriented GEM is connected. This class admits a double SET/SEQ decomposition, and the associated derivative class is the class of indecomposable (multi-)permutations.
The class of general (oriented or not) GEMs, however, does not seem to admit a natural SET/SEQ decomposition.

Let us move on to the description of the model.
Fix an integer $D\geqslant2$, which we will refer to as the dimension.
To obtain a \emph{graph encoded manifold (GEM)} of dimension $D$ of size $n$, we take $n$ simplices of dimension $D$ and glue them together along their hyperfaces.
Here, we assume that the simplices are labeled from $1$ to $n$,
and that the vertices of each simplex are labeled from $1$ to $D+1$.
The gluing is carried out according to the following rule: the hyperfaces identified with each other must be opposed to vertices with identical labels.
Formally, we choose a collection of perfect matchings $\alpha_1,\ldots,\alpha_{D+1}\in S_n$ and identify the hyperface that is opposite to the $k$th vertex of the $i$th simplex with the hyperface opposite to the $k$th vertex of the $\alpha_k(i)$ simplex.
In an equivalent manner, one can think of GEMs of dimension $D$ and size $n$ as $(D+1)$-regular graphs on $n$ vertices labeled from $1$ to $n$, where every edge carries a color from the set $[D+1]$ and every vertex is incident to exactly one edge of each color.

We denote the species of $D$-dimensional GEMs and its connected subspecies, respectively, by $\mcGEM(D)$ and $\mcCGEM(D)$.
From the definition of the model it follows that these species are linked by the relation
 \[
  \mcGEM(D) = \mcE \circ \mcCGEM(D).
 \]
Another implication of the definition is the formula for the number $\gem_{2n}(D)$ of GEMs of dimension $D$ and size $2n$: 
 \[
  \gem_{2n}(D) = \big((2n-1)!!\big)^{D+1}.
 \]
Evidently, there are no GEMs of any odd size.
Our goal is to derive from these two relations the asymptotic probability that a GEM of dimension $D$ consists of a given number of connected components. 

\begin{proposition}\label{prop: gem asymptotics}
 Let $m$ be a fixed positive integer.
 The asymptotic probability that a~random GEM $g\in\mcGEM=\mcGEM(D)$ of size $2n$ has $m$ connected components satisfies
 \[
  \bbP(g\mbox{ has }m\mbox{ connected components})
   \approx
  \sum\limits_{k\ge0}
   \d_{2k,m}(\mcGEM)
    \cdot
   \binom{2n}{2k}
    \cdot
   \left(\dfrac{\big(2(n-k)-1\big)!!}{(2n-1)!!}\right)^{D+1},
 \]
 where
 \begin{align*}
  \d_{k,m}(\mcGEM)
   &
   = 
  \#\big\{(g',g'')\in\mcGEM_i\cdot\mcGEM_{k-i} \mid 0\leqslant i\leqslant k,\,\, \pi_0(g')=m-1,\,\, \pi_0(g'') \mbox{ is even}\big\} \\
   &
   -
  \#\big\{(g',g'')\in\mcGEM_i\cdot\mcGEM_{k-i} \mid 0\leqslant i\leqslant k,\,\, \pi_0(g')=m-1,\,\, \pi_0(g'') \mbox{ is odd}\big\}
 \end{align*}
 and $\pi_0(g)$ is the number of connected components of the GEM $g$.
\end{proposition}
\begin{proof}
 Due to the relation $\mcGEM(D) = \mcE \circ \mcCGEM(D)$, the result comes from Proposition~\ref{prop: p-periodic species asymptotics} applied to the species $\mcA=\mcGEM(D)$ and $\mcF=\mcE$ with $p=2$.
 Thus, the only thing to verify is that the 2-periodic sequence $a_n(D)=\gem_{2n}(D)/(2n)!$ is gargantuan.
 This can be done in several steps.

 First, we prove that $\big(a_n(D)\big)$ is gargantuan for $D=2$.
 In this case, we have
 \[
  a_n(2) 
   =
  \dfrac{\big((2n-1)!!\big)^3}{(2n)!}
   = 
  \dfrac{\big((2n-1)!!\big)^2}{n!} \cdot \dfrac{1}{2^n}.
 \]
 As we have seen in the proof of Proposition~\ref{prop: qss asymptotics}, the sequence $\qss_n/n! = \big((2n-1)!!\big)^2/n!$ is gargantuan.
 Therefore, due to Lemma~\ref{lemma: a_nс^n is gargantuan}, the sequence $\big(a_n(2)\big)$ is gargantuan too.

 Second, we make sure that the sequence $b_n = (2n-1)!!$ is gargantuan.
 To this end, we apply Lemma~\ref{lemma: sufficient conditions for gargantuan sequence}.
 Indeed, its first condition holds, since
 \[
  nb_{n-1} = n(2n-3)!! = O(b_{n}).
 \]
 For the second condition, the sequence $x_k=a_ka_{n-k}$ is decreasing for $k<n/2$, because
 \[
  \dfrac{x_{k+1}}{x_k} = \dfrac{2k+1}{2(n-k)-1} \leqslant 1
 \]
 in this case.
 Thus, the sequence $(b_n)$ is truly gargantuan.

 Finally, the general case comes from the relation $a_n(D) = a_n(2) \cdot b_n^{D-2}$.
 Here, according to Lemma~\ref{lemma: a_nb_n is gargantuan}, we can claim that $\big(a_n(D)\big)$ is gargantuan as a piecewise product of several gargantuan sequences.
\end{proof}

\begin{corollary}\label{cor: gem asymptotics, dominant}
 Let $m$ be a fixed positive integer.
 The dominant term of the asymptotic probability that a random GEM $g\in\mcGEM=\mcGEM(D)$ of size $2n$ has $m$ connected components satisfies
 \begin{equation}\label{formula: gem asymptotics, dominant}
  \bbP(g\mbox{ has }m\mbox{ connected components})
   \sim
  \dfrac{1}{(m-1)!} \cdot \left(\dfrac{1}{2^Dn^{D-1}}\right)^{m-1}.
 \end{equation}
 In particular,
 \[
  \bbP(g\mbox{ is not connected})
   \sim
  \dfrac{1}{2^Dn^{D-1}}.
 \]
\end{corollary}
\begin{proof}
 Since $\gem_2(D)=1$, these relations follow from Corollary~\ref{cor: leading term of p-periodic species asymptotics}:
 \[
  \dfrac{1}{(m-1)!}
   \cdot
  \dfrac{(2n)_{2(m-1)}}{2^{m-1}}
   \cdot
  \left(\dfrac{\big(2(n-m+1)-1\big)!!}{(2n-1)!!}\right)^{D+1}
   \sim
  \dfrac{1}{(m-1)!} \cdot \left(\dfrac{1}{2^Dn^{D-1}}\right)^{m-1}.
 \]
\end{proof}

\section{Conclusion}
\label{sec: conclusion}

We have seen that the structure of an asymptotic expansion of irreducibles determined by a composition $\mcA = \mcF(\mcB)$ arises from the derivative structure $\mcD$ of this composition.
Theorems~\ref{theorem: species E asymptotics} and~\ref{theorem: species L asymptotics} provide an expression for $\mcD$ in terms of the initial species of structures~$\mcA$ in the case where $\mcA$ is gargantuan and $\mcF\in\{\mcE,\mcL,\mcCP\}$.
In general, the derivative structure $\mcD$ is virtual, which is well-illustrated by connected graphs within the Erd\H{o}s-R\'enyi model $G(n,p)$:
the involved coefficients represent tournaments with ties that are purely virtual for $p<1/2$.
On the other hand, in some cases, the asymptotic coefficients are nonnegative, and hence, can be potentially interpreted as total weights of species that are not virtual (and even as counting sequences of some combinatorial classes).
For instance, this is the case for the cycle composition $\mcA=\mcCP(\mcB)$, as well as for the double decomposition $\mcA=\mcE(\mcB)=\mcL(\mcD)$ where irreducibility is understood as connectivity.
Establishing general conditions for the coefficients to be nonnegative is an open problem.

The condition of being gargantuan, imposed on the species $\mcA$, is necessary to apply our theorems.
In the case where this condition does not hold, we cannot ensure that a random object $s\in\mcA$ is irreducible, \emph{i.e.} belongs to $\mcF_1(\mcB)$, with high probability.
Extending our structural approach to the case where the sequence $(\a_n)$ of weights is not gargantuan and the limiting probability differs from $1$ is of great interest.
Potential applications include, for example, the asymptotic probability of various families of trees.

Another setting that we cannot seize by the presented approach concerns such combinatorial structures as biconnected graphs or noncrossing partitions whose generating series are defined implicitly.
Thus, the generating series $C(z)$ and $B(z)$ of connected and biconnected graphs, respectively, are linked by the relation~\cite[formula (1.3.3)]{HararyPalmer1973}
 \[
  \log C(z)=B'(zC'(z)),
 \]
while the (ordinary) generating functions $A(z)$ and $I(z)$ of, respectively, all and noncrossing partitions satisfy~\cite{Beissinger1985}
 \[
  A(z) = 1 + I\big(zA(z)\big).
 \]
To establish asymptotic expansions in these settings, Bender's theorem (Theorem~\ref{theorem: Bender}) is not applicable anymore.
A promising approach was proposed by Borinsky~\cite{Borinsky2018} who assembled the asymptotic coefficients into a new generating function and discovered general rules it obeys, including rules for compositions and inverses.
Unfortunately, Borinsky's method works only for factorially divergent series.
Some ideas of Borinsky were extended for graphically divergent series by Dovgal and one of the authors of this paper~\cite{DovgalNurligareev2022}.
However, the latter approach does not allow one to work with compositions and inverses, and therefore the asymptotics of biconnected graphs and other similar structures cannot be obtained by this way.
Further advances in this direction are still needed.

One final remark we would like to make concerns unlabeled structures.
In general, species theory is well-adapted for a unified study of both labeled and unlabeled combinatorial objects.
In this paper, we discuss only labeled structures, but in principle, in certain cases, our method also works for unlabeled structures; see~\cite[Section~7]{MonteilNurligareevSEQ}.
Creating a unified approach that embraces labeled and unlabeled cases under the same framework is another intriguing open problem.

\section*{Acknowledgments}
Khaydar Nurligareev was supported by the project PICS ANR-22-CE48-0002, as well as the ANR-FWF project PAnDAG ANR-23-CE48-0014-01, both funded by the Agence Nationale de la Recherche.

\bibliographystyle{abbrv}
\bibliography{bibliography}

\end{document}